\documentclass{amsart}
\usepackage{graphicx}
\usepackage{hyperref}
\usepackage{xcolor}
\usepackage{enumitem}
\usepackage[normalem]{ulem}
\usepackage{float}

\theoremstyle{definition}

\theoremstyle{definition}

\theoremstyle{remark}

\numberwithin{equation}{section}
\newcommand\be{\begin{equation*}}
\newcommand\ee{\end{equation*}}
\newcommand\bea{\begin{eqnarray*}}
\newcommand\eea{\end{eqnarray*}}
\newcommand\bi{\begin{itemize}}
\newcommand\ei{\end{itemize}}
\newcommand\ben{\begin{enumerate}}
\newcommand\een{\end{enumerate}}
\newtheorem{thm}{Theorem}[section]

\newtheorem{lem}[thm]{Lemma}
\newtheorem{prop}[thm]{Proposition}

\newtheorem{defi}[thm]{Definition}

\newtheorem{rek}[thm]{Remark}

\newcommand\floor[1]{\left\lfloor#1\right\rfloor}

\newcommand{\R}{\ensuremath{\mathbb{R}}}
\newcommand{\C}{\ensuremath{\mathbb{C}}}
\newcommand{\Z}{\ensuremath{\mathbb{Z}}}
\newcommand{\Q}{\mathbb{Q}}
\newcommand{\N}{\mathbb{N}}

\newcommand{\norm}[1]{\left|\left|#1\right|\right|}
\newcommand{\m}{\textbf{m}}
\newcommand{\w}{\textbf{w}}

\begin{document}

\title[The $RF$-Value for $E(1,a)\hookrightarrow P(\lambda, \lambda b)$]{The Rigid-Flexible Value for Symplectic Embeddings of Four-Dimensional Ellipsoids into Polydiscs}


\author{Alvin Jin}

\address{Massachusetts Intitute of Technology, Cambridge, MA 02139}
\email{alvinjin@mit.edu}

\author{Andrew Lee}
\address{St. Thomas Aquinas College, Sparkill, NY 10976}
\email{alee@stac.edu}


\begin{abstract}

We consider the embedding function $c_b(a)$ describing the problem of symplectically embedding an ellipsoid $E(1,a)$ into the smallest scaling of the polydisc $P(1,b)$. Previous work suggests that determining the entirety of $c_b(a)$ for all $b$ is difficult, as infinite staircases can appear for many sequences of irrational $b$. In contrast, we show that for every polydisc $P(1,b)$ with $b>2$, there is an explicit formula for the minimum $a$ such that the embedding problem is determined only by volume. That is, when the ellipsoid is sufficiently stretched, there is a symplectic embedding of $E(1,a)$ fully filling an appropriately scaled polydisc $P(\lambda,\lambda b)$. Denoted $RF(b)$, this rigid-flexible ($RF$) value is piecewise smooth with a discrete set of discontinuities for $b>2$. At the same time, by exhibiting a sequence of obstructive classes for $b_n = \frac{n+1}{n}$ at $a=8$, we show
that $RF$ is also discontinuous at $b=1$.


\end{abstract}

\maketitle

\tableofcontents

\section{Introduction, Statement of Results}
The problem of embedding one symplectic manifold into another touches on a wide variety of topics in symplectic geometry, and in this work we focus in particular on embeddings of ellipsoids $E(a,b)$ into polydiscs $P(a,b)$ . Here, a polydisc 
\be
P(a,b) := \left\{(z_1,z_2) \in \C^2 \,| \, \pi \|z_1\|^2 < a, \pi \|z_2\|^2 < b \right\} \nonumber
\ee
is the 4-dimensional open symplectic manifold $B_2(a)\times B_2(b)\subset \C^2$, where each factor is a 2-disc of fixed radius centered at $0\in\C$. Similarly the ellipsoid $E(a,b)$ is given by
\be
E(a,b) := \left\{(z_1, z_2)\in\C^2\,|\, \frac{\pi\norm{z_1}^2}{a}+\frac{\pi\norm{z_2}^2}{b} < 1\right\}.\nonumber
\ee
We study the problem of finding embeddings of an ellipsoid into a polydisc. This information is encoded by a function $c_b(a)$ whose value at $a$ is the smallest $\lambda$ such that 
\begin{equation}
E(1,a)\hookrightarrow P(\lambda, \lambda b).\label{embeddingproblem}
\end{equation}
More precisely,
\be
c_b(a):= \inf \{ \lambda > 0 \,|\, E(1,a)\hookrightarrow P(\lambda, \lambda b)\}.\nonumber
\ee
Earlier work on embeddings of ellipsoids into polydiscs \cite{BPT, FrMu} considered a fixed value of $b$ and calculated $c_b(a)$ in its entirety for all $a$. For example, $c_1(a)$ contains a so-called ``infinite staircase," a convergent sequence of $a$-values $a_n\to a_{\infty}$ such that $c_b(a)$ is non-decreasing and linear or constant on each $[a_i, a_{i+1}]$. In addition, they find that for sufficiently large values of $a$, $c_b(a)$ coincides with the volume constraint $\sqrt{\frac{a}{2b}}$: for sufficiently elongated ellipsoids the only obstruction to the existence of symplectic embeddings is the volume constraint.

In later work \cite{CFS, Ush} both the source and target in the embedding problem were allowed to become elongated, with the goal of analyzing the resulting changes in the embedding function $c_b(a)$.

To set notation, define
\begin{equation}
RF(b): = \inf\left\{A\,\,|\,\, c_b(a) = \sqrt{\frac{a}{2b}}\,\,\mbox{for all}\,\, a\geq A\right\}.\nonumber
\end{equation}
We call this number the \emph{rigid-flexible value}, or the $RF$-value. \newline\indent

Varying this extra parameter $b$ uncovers delicate structure in the the $b$-direction for functions $c_b(a)$. In \cite{CFS} the authors find that for $b\in\N$ with $b\geq 2$, $c_b(a)$ contains no infinite staircases and can be computed explicitly for all $a$. On the other hand, \cite{Ush} provides examples of infinite sequences (some arbitrarily large and others converging to 1) of irrational $b$ such that for each such $b$, $c_b(a)$ contains an infinite staircase. Though it is possible to compute the embedding functions $c_b(a)$ for $b\in [1,\infty)$ and $a\in [1,3+\sqrt{2}]$, the rest is unknown in full generality.


These results show that one cannot necessarily ``interpolate" to deduce values of $c_b(a)$ from what is known for $b\in Z$. 
The complicated structure in the embedding function $c_b(a)$ cannot be described as straightforwardly as in \cite{CFS}, where $c_b(a)$ for all $a$ is described with explicit formulas. 
Thus, in the context of other work on this subject, our results for $b>2$ provide a counterpoint to the complexity in \cite{Ush}. The computation of the $RF$-value appears to be a tractable problem, akin to the computation in \cite{CFS} of the embedding functions $c_b(a)$ for integral $b$ and the computation in \cite{Ush} for $a\in [1,3+2\sqrt{2}]$.

In this paper, we focus on the $RF$-value in the polydisc problem \ref{embeddingproblem} for non-integral $b$.

\begin{thm}\label{theoremA}
For $b>2$, $RF(b)$ is given by
\begin{equation}
RF(b) = 2b\left(\frac{2\floor{b+\sqrt{2b}}+1}{b+\floor{b+\sqrt{2b}}}\right)^2.\label{largebRF}
\end{equation}
\end{thm}
As a consequence, the function $RF(b)$ is increasing and piecewise smooth in $b$, with a discrete set of discontinuities.
%
However, the behavior of $RF(b)$ is more delicate for $1<b<2$. In the case $b=1$, \cite{FrMu} shows that $RF(1)=7\frac{1}{32}$
whereas \cite{CFS} shows that $RF(2) = 8\frac{1}{36}$.
In particular $RF(2)$ does not follow the formula above, so already the function $RF:[1,\infty)\to\R$ is not continuous at $b=2$.

In addition, the result below establishes that $RF(b)$ also fails to be continuous at $b=1$. We show that for the sequence $b_n = \frac{n+1}{n}$, we have $\lim_{n\to\infty} RF(b_n)\neq RF(1)$.
\begin{thm}\label{theoremB}
For the sequence $b_n=\frac{n+1}{n}$, with $n\geq 5$, 
\begin{equation}
RF(b_n) = \frac{2(n+1)}{n}\left(\frac{8n^2+8n+1}{2(2n+1)(n+1)}\right)^2.\label{bnRFvalue}
\end{equation}

\end{thm}
When $n\to\infty$, (\ref{bnRFvalue}) approaches 8, whereas  \cite{FrMu} shows that $RF(1) = 7\frac{1}{32}$.

\subsection{Outline of Paper}
In Sections 3, 4, and 5, we give a proof of Theorem 1.1 by restricting to two distinct intervals of $a$. We first consider values of $a$ less than (\ref{largebRF}) with a positivity of intersection argument, communicated to us by the authors of \cite{CFS}, which suffices to show that $c_b(a)$ on this interval is strictly above the volume constraint until the value of $a$ in (\ref{largebRF}). Then, as Proposition 3.5 (ii) in \cite{CFS} establishes that $RF(b)\leq (\sqrt{2b}+1)^2$, we use the reduction method to show that on the interval from (\ref{largebRF}) to $a=(\sqrt{2b}+1)^2]$, $c_b(a)$ equals the volume constraint.

On this second interval, the proof is again split into two sections, each organized according to the possible orderings of terms in a certain vector. Although this choice of exposition emphasizes that our argument is comprehensive, it also leads to a proliferation of cases.

Then in Section 6, we proceed with the proof of Theorem 1.2. We first establish an upper bound $RF(b_n)\leq 9$ using the reduction method. We then calculate $c_{b_n}(8)$ by exhibiting a sequence of obstructive classes $R_n$ which determine $c_{b_n}(8)$. An argument similar to that in \cite[Section 3.6]{CFS}  then establishes that the only possible obstructive classes on $a\in [8,9)$ are the $R_n$. 

\subsection{Acknowledgments}
We thank Dan Cristofaro-Gardiner for suggesting this problem, and for both his and Felix Schlenk's patience in explaining the arguments in \cite{CFS}. We also thank the anonymous referee for their feedback which greatly improved this work.

\section{Preliminaries: Two Methods for Finding Symplectic Embeddings}
Here, we review the methods we use for detecting symplectic embeddings. This is not an exhaustive list, and more detailed expositions are in \cite{CFS,McSc}, so we review only what we use. The following definition is central to both methods. Fix $b\geq 1$. Since the function $c_b(a)$ is continuous in $a$, it suffices to compute it for $a \geq 1$ rational.
\begin{defi}
The \textbf{weight expansion} $\w(a,b)$ of such an $a$ is the finite decreasing sequence
\be
\w(a) = (1^{\times \ell_0}, w_1^{\times \ell_1},...,w_n^{\times \ell_n})\nonumber,
\ee
where $w_1 = a - \ell_0 < 1, w_2 = 1 - \ell_1 w_1 < w_1$, and so on.

We write $\w(a)$ to mean $\w(1,a)$, and we write $w(a_i)$ to mean the $i$th term of the weight expansion $\w(a)$.
\end{defi}
In \cite[Thm. 1.1]{McD}, it is shown that there is a canonical decomposition of any ellipsoid $E(1,a)$ into a collection of balls
\be
E(1,a) := \coprod_i B(w_i)\nonumber
\ee
where the $w_i$ are the terms in the weight expansion $\w(a)$.
\begin{thm}[{\cite{FrMu}}]
Let $a,b\in\Q$ be positive. There exists a symplectic embedding $E(1,a)\hookrightarrow P(\lambda, \lambda b)$ if and only if there is a symplectic embedding
\be E(1,a)\coprod B(\lambda)\coprod B(\lambda b) \hookrightarrow B(\lambda(b+1)).\nonumber\ee
\end{thm}
This reduces the polydisc problem to a ball-packing problem of embedding balls of capacity $e_i$ into a ball of capacity $\mu$:
\begin{equation}
\coprod_i B(e_i)\hookrightarrow B(\mu).\label{ballembedding}
\end{equation}
The purpose of introducing these constructions is that a collection of balls can be embedded symplectically precisely when the associated multiple blow-up of $\mathbb{C}P^2$ carries a symplectic form in a certain cohomology class. We denote by $\overline{\mathcal{C}}_K(X_n)$ the set of cohomology classes represented by symplectic forms for which the anticanonical class is $K=-3L+\sum_i E_i$, where $L$ is Poincar\'e dual to a line in $\mathbb{C}P^2$ and each $E_i$ is dual to the $i$th exceptional sphere. By \cite{McPo}, the embedding (\ref{ballembedding}) exists when the following cohomology class is in the symplectic cone:
\be
\mu L - \sum_i e_i E_i \in \overline{\mathcal{C}}_K(X_n).\newline\indent
\ee
The above fact gives a sufficient criterion for a class to lie in $\overline{\mathcal{C}}_K(X_n)$. If there is a symplectic form in a given class, it must have non-negative intersection with certain holomorphic curves. By \cite{LiLiu}, this is also sufficient: if $\mathcal{E}_K(X_n):=\{e\in H_2(X_n)\,|\,  \langle e, e\rangle = -1,\,\,\langle K, e\rangle = -1\}$ is the set of exceptional classes, then we may characterize the symplectic cone referenced above as
\begin{equation}
\overline{\mathcal{C}}_K(X_n)=\{\omega\in H^2(X_n)\,|\,  \langle \omega, e\rangle\geq 0\,\, \forall e\in \mathcal{E}_k(X_n)\}. \label{symplecticcone}
\end{equation}
In summary, an embedding of the form (\ref{ballembedding}) exists when the associated weight expansion represents a cohomology class in the symplectic cone $\overline{\mathcal{C}}_K(X_n)$, and the equality (\ref{symplecticcone}) characterizes elements of the symplectic cone as those which pair non-negatively with so-called exceptional classes $\mathcal{E}_k(X_n)$.

We now require a condition for determining when homology classes are exceptional. With respect to the basis of $H^2(X_n;\R)\simeq \R^{n+1}$ given above, a \emph{Cremona transform} is the map given by
\be
(d; m_1,...m_n)\mapsto (2d-m_1-m_2-m_3;d-m_2-m_3, d-m_1-m_3, d-m_1-m_2, m_4,...,m_n)\nonumber
\ee
We will use Cremona transforms in two different contexts below. First, we can state the condition for an integral homology class to lie in $\mathcal{E}_k$, which is proven in \cite{McSc} based on work of \cite{LiLi, LiLiu}. 
\begin{thm}
A class $(d;m_1,...,m_n)\in H_2(X_n;\Z)$ is in $\mathcal{E}_K(X_n)$ if and only if its entries satisfy the Diophantine equations
\begin{eqnarray}
& & 3d-1 = \sum_i m_i \label{diophantine} \newline\\
& & d^2+1 = \sum_i m_i^2 \nonumber
\end{eqnarray}
and $(d;m_1,...,m_n)$ reduces to $(0;-1,0, \ldots, 0)$ after a sequence of  Cremona transforms.
\end{thm}

For the problem of embedding ellipsoids into polydiscs, the natural compactification of $P(a,b)$ adds a single point to each disc, yielding $S^2\times S^2$. This manifold is in fact diffeomorphic to the 2-fold blow-up of $\mathbb{C}P^2$, so the $n$-fold blow-up of $S^2\times S^2$ (denoted $Y_n$) can be identified with $X_{n+1}$. The induced isomorphism on cohomology $\psi:H^2(Y_n)\to H^2(X_{n+1})$ is given by
\be
(d,e;m_1,m_2,m_3,...,m_n)\mapsto (d+e-m_1; d-m_1, e-m_1, m_2,...,m_n)
\ee

Using this isomorphism we may translate the Diophantine equations (\ref{diophantine}) into the following conditions.
\begin{eqnarray}
\sum_i m_i & = & 2(d+e) -1 \label{dioph1}\newline\\
\sum_i m_i^2 & = & 2de+1 \label{dioph2}
\end{eqnarray}



\subsection{Obstructive Classes}\label{method1}
We can now describe the first of the two methods we employ to find symplectic embeddings. 
If $A := (d,e;\textbf{m})\in \mathcal{E}_k$ with the basis given previously, the condition
\begin{equation}
\lambda = \frac{\langle \textbf{w}(a), \textbf{m}\rangle}{d+be} >0 \label{lambda}
\end{equation}
is equivalent to
\begin{equation}
\lambda(d+be) - \langle \textbf{w}(a), \textbf{m}\rangle >0.
\end{equation}

The following statement is \cite[Method 1']{CFS}.

\begin{thm}\label{classcondition}
An embedding (\ref{embeddingproblem}) exists iff $\lambda\geq \sqrt{\frac{a}{2b}}$ and 
\be
\mu_b(d,e;m):= \frac{\sum_i m_i \cdot w_i(a)}{d+be} \leq \lambda\label{obstr}
\ee
for every $(d,e;m_1,...,m_n)\in H_2(Y_n;\Z)$ which satisfies equations (\ref{dioph1}, \ref{dioph2}) and reduces to $(0;-1)$ after some sequence of Cremona transforms.
\end{thm}

\begin{defi}
We say that a class $A$ is \textbf{obstructive at $a>0$} if the quotient in (\ref{lambda})
is larger than the volume constraint $\sqrt{\frac{a}{2b}}$.
\end{defi}
To use this method, then, we must find all obstructive classes. In practice, this method proceeds first by finding a discrete subset of $a$ which have a special relationship to the obstructive classes, described as follows.

If the graph of $c_b(a)$ does not follow the volume constraint, then locally it must be given by the obstruction function of some class $(d,e;\m)$. Restricting to the interval where this class determines the graph, \cite{McSc} Lemma 2.1.3 states that there is a particular $a$-value on the interval whose weight expansion coincides with the number of positive entries of $\m$, the tail of the obstructive class.

\begin{lem}\label{lengthbound}
Let $C = (d,e;\textbf{m})$ be an exceptional class, and let $I$ be the maximal nonempty open set on which $\mu_{a,b}(C)>\sqrt{\frac{a}{2b}}$ for all $a\in I$. Then, there is a unique element $a_0\in I$ such that the length of the weight expansion for $a_0$, denoted $\ell(a_0)$, satisfies $\ell(a_0) = \ell(\textbf{m})$, and moreover $\ell(a)>\ell(\textbf{m})$ if $a\in I\setminus\{a_0\}$.
\end{lem}

The following quantities will be relevant going forward: let $\ell_0$ denote the number of $1$'s in the weight expansion of $a$ and subsequently let $\ell_i$ denote the lengths of subsequent blocks. When a class $(d,e;\w(a))$ is obstructive, we have a vector of error terms $\epsilon$ defined as
\be
\m = \frac{d+be}{\sqrt{2a}}\w(a) + \epsilon.
\ee
Each contribution to this error, thought of as the difference between $\m$ and $\frac{d+be}{\sqrt{2a}}\w(a)$, will be important, so we define $v_i = \frac{d+be}{\sqrt{2a}}w_i$ for $i=0,...,M$ where $M$ denotes the length of the weight expansion of $a$.\newline\indent
Also let
\be
\sigma = \sum_{\ell_0 + 1}^M \epsilon_i \nonumber
\ee
denote the ``residual error," the contribution to the error vector coming from non-integer terms of $\w(a)$. Also, define a related quantity 
\be
\sigma' = \sum_{\ell_0 + 1}^{M-\ell_N} \epsilon_i \nonumber.
\ee
Here $\ell_N$ is the length of the last block. This quantity ignores the contribution to the error from the smallest part of the weight expansion $\w(a)$.  \newline\indent
With these terms established, the following results (proven in \cite{CFS,McSc}) will be used repeatedly.

\begin{lem}\label{degfacts}
For all $(d,e;\m)\in\mathcal{E}$, suppose that $a\in\Q$ and $b\in\R$ with $a,b>1$.
Then we have
\begin{enumerate}
    \item $\mu_b(d,e;\m)(a)\leq\sqrt{a}\sqrt{\frac{2de+1}{(d+be)^2}}$\label{errorsumlesssigmal}
    \item $\mu_b(d,e;\m)(a)>\sqrt{\frac{a}{2b}}$ if and only if $\epsilon\cdot\w>0$\label{errorsumlesssigmapl}
    \item If $\mu_b(d,e;\m)(a)>\sqrt{\frac{a}{2b}}$ then $d=be+h$ where $|h|<\sqrt{2b}$ and $\langle \epsilon, \epsilon\rangle<1-\frac{h^2}{2b}$\label{errorbound}
    \item Let $y(a) = a+1-2\frac{b+1}{\sqrt{2b}}\sqrt{a}$, where $a=p/q$ is rational. Then \begin{equation} -\sum_i\epsilon_i = 1+\frac{d+be}{\sqrt{2ab}}\left(y(a)-\frac{1}{q}\right)\label{errorsumformula}\end{equation}
\end{enumerate}
\end{lem}
\begin{proof}
The proofs of (1)-(3) are as in \cite[Prop. 3.1]{CFS}. A more general version of (4) is given in Equation (4.4) in \cite{CHPM}. The form appearing here follows from choosing the convex toric domain to be the polydisc $P(\lambda, \lambda b)$.
\end{proof}

With these prerequisites in hand, the second step of the obstructive class method establishes that in fact only finitely many classes must be checked at the values of $a$ characterized in Lemma \ref{lengthbound}. This is the purpose of the inequality (\ref{ebound}) below.

\begin{lem}\label{errorestimates}
Let $(d,e;\textbf{m})$ be an exceptional class such that $\ell(a) = \ell(m)$ and $\mu_b(d,e;\m) > \sqrt{\frac{a}{2b}}$, and $b\in[1,2]$. Set $v_M := \frac{(d+be)\sqrt{2b}}{q(b+1)\sqrt{a}}$ where $q$ is the last denominator in the weight expansion $\w(a)$. Then
\begin{enumerate}
    \item $\left|\sum_i \epsilon_i\right| \leq \sqrt{\sigma L}$ 
    \item If $v_M<1$, then $\left| \sum_i \epsilon_i\right|\leq \sqrt{\sigma'L}$, 
    \item If $v_M\leq \frac{1}{2}$, then $v_M> \frac{1}{3}$ and $\sigma'\leq \frac{1}{2}$. If $v_M\leq \frac{2}{3}$, then $\sigma'\leq \frac{7}{9}$.\label{vmsigmabounds}
    \item With $\delta = y(a)-\frac{1}{q}$ and $y(a) = a+1-2\frac{b+1}{\sqrt{2b}}\sqrt{a}$, we have \label{qbound}
    \begin{equation}\label{qboundformula}
    \sqrt{q + \floor{a} + 1} \geq 1 + \delta v_M q
    \end{equation}
One can also show the bound from \cite{CFS, FrMu}, which is
\begin{equation}
2be+h \leq \frac{\sqrt{2ba}}{\delta}\left(\sqrt{\sigma q} - (1-h(1-\frac{1}{b})\right) \leq \frac{\sqrt{2ba}}{\delta}\left(\frac{\sigma}{\delta v_M}-\left(1-h(1-\frac{1}{b}\right)\right) \label{ebound}
\end{equation}
\end{enumerate}
\end{lem}

\begin{proof}[Proof of Lemma \ref{errorestimates}]
(1) and (2) are as in \cite[Lemma 5.1.2]{McSc}. Though Claim (3) involves terms which depend on $b$, the same proof for fixed $b$ goes through as in \cite[Lemma 5.1.2 (iii)]{McSc} and \cite[Lemma 3.7 (iii)]{CFS}.

To show (4), we have from Equation (\ref{errorsumformula}) that
\be
-\sum_i \epsilon_i =  \frac{2be+h}{\sqrt{2ba}}(a+1-\frac{b+1}{b}\sqrt{2ba}-\frac{1}{q}) + (1-h(1-\frac{1}{b})\nonumber
\ee

Using $q\geq 1\geq \sigma$ and (1) of this lemma,
\bea
\sqrt{qL} &\geq & \sqrt{\sigma L}\newline\\
& \geq & \frac{2be+h}{\sqrt{2ba}}(a+1-\frac{b+1}{b}\sqrt{2a}-\frac{1}{q}) + (1-h(1-\frac{1}{b})\newline\nonumber\\
& = & \frac{2be+h}{\sqrt{2ba}}\delta + (1-h(\frac{1}{b}))\newline\nonumber\\
& > & \delta v_M q\nonumber
\eea
where the first line is by (1), the second is by definition of $\delta$, and the last comes from $b\in[1,2)$.\newline\indent
It follows that $\sqrt{q}<\frac{\sqrt{\sigma}}{\delta v_M}$. Rearranging the above inequality we see
\bea
2be+h & \leq & \frac{\sqrt{2ba}}{\delta}\left(\sqrt{\sigma q} - (1-h(1-\frac{1}{b}))\right)\newline\nonumber\\
& < & \frac{\sqrt{2ba}}{\delta}\left(\frac{\sigma}{\delta v_M} - (1-h(1-\frac{1}{b})\right)
\eea
\end{proof}

The following lemma, \cite[Lem. 2.1.8]{McSc} will also be necessary for our arguments.
\begin{lem}\label{oneblock}
Assume that $(d,e;\m)$ is an exceptional class such that $\mu(d,e;\m)>\sqrt{\frac{a}{2b}}$ for some $a$. Let $J:=\{k,...,k+s-1\}$ be a block of $s-1$ consecutive integers for which $w(a_i), i\in J$ is constant. Then $(m_1,...,m_{k+1})$ is of the form
\bea
(m,...,m)\newline\nonumber\\
(m-1,m,...,m)\newline\nonumber\\
(m,...,m+1).\nonumber
\eea
Moreover, there is at most one block of length $s\geq 2$ on which the $m_i$ are not all equal, and if $m_1\neq m_{k+1}$, then $\sum_{i=k}^{k+s-1}\epsilon_i^2 \geq \frac{s-1}{s}$.
\end{lem}

\subsection{Reduction at a Point}\label{method2}
Though the first method is a necessary and sufficient condition for the embedding of an ellipsoid into a polydisc, in principle one might need to check more exceptional classes than is computationally feasible. The following method straightforwardly determines whether an embedding exists but the complexity of the process depends strongly on the value of $a$.\newline\indent

\begin{defi}
The \textbf{defect} $\delta$ of an ordered vector $(d;m_1,...,m_n)$ is the sum $d-m_1-m_2-m_3$.
\end{defi}
The following is established in \cite{BuPi, KaKe}.
\begin{thm}
An embedding $E(1,a)\hookrightarrow P(\lambda, \lambda b)$ exists if there exists a finite sequence of Cremona moves that transforms the ordered vector
\be
((b+1)\lambda; b\lambda, \lambda, \w(a))\nonumber
\ee
to an ordered vector with non-negative entries and defect $\delta\geq 0$.
\end{thm}
We will apply this repeatedly in the proofs of both Theorem 1.1 and Theorem 1.2.

Notice the requirement that the Cremona moves be performed on ordered vectors. The following fact \cite[Prop. 2.2]{CFS} allows us to avoid re-ordering.

\begin{prop}
Let $\alpha=(\mu;a_1,...,a_n)$ be a vector with $\mu\geq 0$ and $\alpha^2\geq 0$ and assume that there is a sequence $\alpha = \alpha_0,...,\alpha_m$ of vectors such that $\alpha_{j+1}$ is obtained from $\alpha_j$ by a sequence of Cremona moves. If $\alpha_m = (\hat{\mu};\hat{\alpha}_1,...,\hat{\alpha}_n)$ is reduced and $\hat{\alpha}_1,...,\hat{\alpha}_n\geq 0$, then $\alpha\in\overline{\mathcal{C}}_K(X_n)$.
\end{prop}

Thus to find embeddings we need only prove enough about the ordering to ensure that the defect at a certain step is non-negative, provided that we have non-negativity of the terms in each vector. So, only knowledge of the ordering for the first three terms is strictly necessary. We will usually verify the ordering at each step, except when doing so results in distinguishing too many cases.

\section{Positivity of intersection: $a=2n_b+1$ and $b>2$}\label{2nb+1}
In this section we begin examining the $RF$-value in the case of non-integers, recalling from \cite{CFS} the classes which determine the $RF$-value when $b>2$. We thank the authors of \cite{CFS} for communicating the argument of this section, which was used initially for $b\in\N$.

In what follows, let
\be
n_b =\floor{\sqrt{2b}+b}.\label{nb}
\ee
The associated obstructive class is given by 
\be
(n_b, 1; 1^{\times 2n_b+1}),\nonumber
\ee
which corresponds to the class $E_{n_b}$ in \cite{CFS}. In this notation, we may restate Theorem \ref{theoremA} as $RF(b) = 2b\cdot \mu_b(E_{n_b})^2$.
 \newline\indent

\subsection{Calculations when $a=2n_b+1$}
We show in this section using the arguments of Section \ref{method1} that the $E_n$ are the only obstructions relevant when $b> 2$ and $a=2n_b+1$.

The key fact we use is positivity of intersection for the homology classes $\mathcal{E}$. As they are represented by symplectically embedded spheres, the intersection of any two distinct classes must be positive. This general fact is stated in \cite[Prop. 3.8 (ii)]{FrMu}.
\begin{prop}
For all distinct $(d,e;\textbf{m})\in \mathcal{E}$ we have
\begin{equation}
\sum_i m_i m_i'\leq de'+d'e.
\end{equation}
\end{prop}

\begin{prop}
For $b>2$, the only obstructive class at $a=2n_b + 1$ is $E_{n_b}$.
\end{prop}

\begin{proof}
Suppose that there is an obstructive class $(d,e;\textbf{m})$ other than $E_{n_b}$ at $2n_b+1$. Then it must have positive intersection with $E_{n_b}$, so
\be
0 \geq d + n_b e - \langle \textbf{m}, 1^{\times 2n_b+1}\rangle\nonumber 
\ee
and hence $\langle \textbf{m}, 1^{\times 2n_b+1}\rangle \leq d+n_b e$.
Then we can bound its obstruction function above as follows:
\bea
\mu(d,e;\textbf{m})(2n_b+1) & = & \frac{\langle \textbf{m}, 1^{\times 2n_b+1}\rangle}{d + b e}\nonumber\\
& = & \frac{\sum_i m_i}{d+b e}\nonumber \\
& \leq & \frac{d+n_b e}{d+b e}\nonumber \label{keybound}
\eea
We wish to show that this is no larger than $\frac{2n_b+1}{b + n_b}$, the constraint arising from $E_{n_b}$. Subtracting one from both sides of the inequality so $d$ appears only once, this amounts to
\bea
\frac{d+n_b e}{d+b e} & \leq & \frac{2n_b+1}{b + n_b} \nonumber\\
\frac{(n_b-b)e}{d+b e} & \leq & \frac{n_b-b+1}{n_b + b}\nonumber\\
(n_b+b)(n_b-b)e & \leq & (b e+d)(n_b-b+1)\nonumber
\eea
When $d\geq b e$, we see that the right-hand side can be bounded below, so it suffices to show that
\bea
(n_b+b)(n_b-b)e & \leq & 2b e(n_b-b+1)\nonumber\\
(n_b+b)(n_b-b) & \leq & 2b (n_b-b+1)\nonumber\\
n_b^2 - b^2 & \leq & 2b n_b - 2b^2 + 2b\nonumber\\
n_b^2 - 2b n_b +b^2 & \leq & 2b\nonumber \\
(n_b -b)^2 & \leq & 2b\nonumber\\
n_b-b & \leq & \sqrt{2b}\nonumber\\
n_b & \leq & b + \sqrt{2b}\nonumber
\eea
which follows immediately from the definition of $n_b$. Hence, we may assume going forward that $d<b e$.

Since the weight expansion of $a=2n_b+1$ is $1^{\times 2n_b+1}$, then by Lemma \ref{oneblock} there are only 3 possibilities for $\textbf{m}$.
\begin{equation}
\textbf{m}=\begin{cases}
M^{\times 2n_b +1 } \\
(M+1, M^{\times 2n_b}) \\
(M^{\times 2n_b}, M-1)
\end{cases}
\label{Mcases}
\end{equation}
By Lemma \ref{lengthbound}, $\textbf{m}$ has length no larger than $2n_b+1$. If $\textbf{m}$ has length less than $2n_b+1$, the last entry is 0, so $M=1$. But then $\sum_i m_i = 2n_b$ while $\sum_i m_i = 2(d+e)-1$. Hence $M>0$, and the length $\textbf{m}$ is exactly $2n_b+1$.

We start with the first case in (\ref{Mcases}). Using the Chern number 1 condition, we obtain
\be
((2n_b+1)M)^2 = (2(d+e)-1)^2 = 4(d+e)^2-4(d+e)+1.\nonumber
\ee
On the other hand, $(2n_b+1)M^2= 2de+1$ by the self-intersection condition. Substituting this in the left-hand side,
\bea
(2n_b+1)(2de+1) & = & 4(d+e)^2-4(d+e)+1\\
(2n_b+1)(2de+1) +4d + 4e -1 & = & 4d^2+8de+4e^2.\nonumber
\eea
Dividing both sides by $4de$, the result is
\begin{eqnarray}
(2n_b+1)\frac{(2de+1)}{4de} +\frac{1}{e} + \frac{1}{d} -\frac{1}{4de} & = & \frac{d}{e}+2+\frac{e}{d}\nonumber\\
\frac{2n_b+1}{2} + \frac{2n_b+1}{4de} +\frac{1}{e} + \frac{1}{d} -\frac{1}{4de}  & = & \frac{d}{e}+2+\frac{e}{d}\nonumber\\
\frac{2n_b+1}{2} + \frac{2n_b}{4de} +\frac{1}{e} + \frac{1}{d} & = & \frac{d}{e}+2+\frac{e}{d}\nonumber\\
n_b + \frac{1}{2} + \frac{n_b}{2de} +\frac{1}{e} + \frac{1}{d} & = & \frac{d}{e}+2+\frac{e}{d}\nonumber\\
n_b + \frac{n_b}{2de} +\frac{1}{e} + \frac{1}{d} & = & \frac{d}{e}+ \frac{3}{2}+\frac{e}{d}\label{fracident}
\end{eqnarray}
We will return to this identity later.\\

Now, we use the assumption that $d<b e$. If $(d,e;\textbf{m})$ is obstructive, then \ref{errorbound} gives the identity $d=b e + h$ with $|h|<\sqrt{2b}$. So write $d=b e-\nu$ with $\nu\in(0,\sqrt{2b})$. 

By the proof of \cite[Lemma 4.1]{Ush}, we may assume without loss of generality that $d\geq e$. Combining this with the assumption that $b>2$ gives the following.
\begin{eqnarray}
2e^2 & < & be^2\nonumber\\
& < & be^2 + 2\nu(d-\frac{e}{2})\nonumber\\
& = & 2\nu d + ed\label{nuinequal}
\end{eqnarray}
Specifically, the second inequality relies on the fact that $2\nu(d-\frac{e}{2})$ is positive. The term $2\nu d - e\nu$ is positive when $2\nu d > e\nu$, and since $\nu>0$ this holds when $2d > e$ which follows from $d\geq e$.

Divide (\ref{nuinequal}) by $2de$ to obtain
\be
\frac{e}{d} < \frac{\nu}{e}+\frac{1}{2}
\ee
which means that the right-hand side of (\ref{fracident}) can be bounded above by
\begin{eqnarray}
n_b + \frac{n_b}{2de} +\frac{1}{e} + \frac{1}{d} & = & \frac{d}{e}+ \frac{3}{2}+\frac{e}{d}\nonumber\\
& < & \frac{d}{e}+\frac{3}{2}+\frac{\nu}{e}+\frac{1}{2}\nonumber\\
& = & \frac{d+\nu}{e}+2\nonumber \\
& = & \frac{b e}{e}+2\nonumber\\
& = & b+2. \label{posintineq}
\end{eqnarray}
The inequality $n_b < b+2$ is false already when $b\geq 4$, so we need only show (\ref{posintineq}) cannot hold for $2\leq b <4$. In these cases, the fact that $b$ is small means we have an explicit bound 
\be
n_b + \frac{n_b}{2de} + \frac{1}{e}+\frac{1}{d} < 4.
\ee
The possible values of $n_b$ in this range are $k=\{4,5\}$. It is then straightforward to check that there are no solutions to this inequality when $d,e\geq 0$. For example, if $k=5$,
\bea
5 + \frac{5}{2de} + \frac{1}{e}+\frac{1}{d} & < & 4\\
\frac{5}{2de} + \frac{1}{e}+\frac{1}{d} & < -1 \\
5 + d + e & < & 2de(-1)
\eea
which is impossible for positive $d,e$.

Now we deal with the second case in (\ref{Mcases}). The resulting Chern number and self-intersection conditions are
\bea
(2n_b+1)M + 1 & = & 2(d+e)-1 \\
(2n_b+1)M^2 + 2M + 1 & = & 2de+1
\eea
Once again we substitute one equation into the other to eliminate $M$, then rearrange. Squaring the Chern number condition, notice that
\bea
((2n_b+1)M + 1)^2 & = & (2n_b+1)^2M^2 + 2(2n_b+1)M + 1\\
& = & (2n_b+1)\left((2n_b+1)M^2 + 2M + \frac{1}{2n_b+1}\right)= (2(d+e)-1)^2 \\
(2n_b+1)M^2 + 2M & = & 2de 
\eea
Making the substitution of terms containing $M$ we obtain 
\bea
(2n_b+1)\left(2de+\frac{1}{2n_b+1}\right) & = & 4(d+e)^2-4(d+e)+1\\
(2n_b+1)(2de) +4d + 4e & = & 4d^2+8de+4e^2
\eea
Then, dividing both sides by $4de$, the result is
\bea
(2n_b+1)\frac{(2de)}{4de} +\frac{1}{e} + \frac{1}{d} & = & \frac{d}{e}+2+\frac{e}{d}\\
\frac{2n_b+1}{2} +\frac{1}{e} + \frac{1}{d}  & = & \frac{d}{e}+2+\frac{e}{d}\\
n_b +\frac{1}{e} + \frac{1}{d} & = & \frac{d}{e}+\frac{3}{2}+\frac{e}{d}
\eea

Now, when the entries of $\textbf{m}$ in the same block are not identical, there is a strong lower bound on error arising from (\ref{ebound}). In this case we see that
\be
\frac{2n_b}{2n_b+1} < \sum_i \epsilon_i^2 
\ee
We also have by \ref{errorbound} that 
\be
\sum_i \epsilon_i^2 < 1-\frac{h^2}{2b}
\ee
so that overall
\bea
\frac{2n_b}{2n_b+1} & < & 1-\frac{h^2}{2b}\\
|h| < \sqrt{\frac{2b}{2n_b +1}}.
\eea
This quantity is bounded above by 1 (equivalently $2b<2n_b+1$, which is easy to show). As a result, we obtain a much stronger bound than the usual $\sqrt{2b}$.

Applying this fact in our case, we see that if $(d,e;\textbf{m})$ is obstructive, then \cite[Prop. 3.1(iii)]{CFS} gives the identity $d=be + h$ but now with $|h|<1$. So write $d=be-\nu$ with $\nu\in(0,1)$.
Without loss of generality $d\geq e$, and throughout we have used $b\geq 2$. Combining all these gives the following.
\bea
2e^2 & < & be^2\\
& < & be^2 + 2\nu(d-\frac{e}{2})\\
& = & 2\nu d + ed
\eea
The second inequality relies on the fact that $2\nu(d-\frac{e}{2})$ is positive. The term $2\nu d - e\nu$ is positive when $2\nu d > e\nu$, and since $\nu>0$ this holds when $2d > e$. As we assumed $d\geq e$, this is true.

Divide this inequality by $2de$ to obtain
\be
\frac{e}{d} < \frac{\nu}{e}+\frac{1}{2}
\ee
which means that the right-hand side of the first identity can be bounded above by
\bea
n_b + \frac{n_b}{2de} +\frac{1}{e} + \frac{1}{d} & = & \frac{d}{e}+ \frac{3}{2}+\frac{e}{d}\\
& < & \frac{d}{e}+\frac{3}{2}+\frac{\nu}{e}+\frac{1}{2}\\
& = & \frac{d+\nu}{e}+2 \\
& < & \frac{d+1}{e}+2\\
n_b + \frac{n_b}{2de}+\frac{1}{e}+\frac{1}{d}& < & \frac{d+1}{e}+2.
\eea
Again using $d<be + 1$, we get
\bea
n_b + \frac{n_b}{2de}+\frac{1}{e}+\frac{1}{d}& < & \frac{b e +1+1}{e}+2\\
n_b + \frac{n_b}{2de}+\frac{1}{e}+\frac{1}{d} & < & b + \frac{1}{e}+ \frac{1}{e}+2\\
n_b\left(1 + \frac{1}{2de}\right)+\frac{1}{d} & < & b + \frac{1}{e}+2\\
\eea
If this inequality held, recall that $d,e\geq 2$ for if $e=1$ then the class must be $E_n$. Then certainly
\bea
\frac{3}{2}n_b & < & b+2.5
\eea
but this is already false once $b>2$.

For the last case, we again work through the identities. The resulting Chern number and self-intersection conditions are
\bea
(2n_b+1)M - 1 & = & 2(d+e)-1 \\
(2n_b+1)M^2 - 2M + 1 & = & 2de+1
\eea
Once again we substitute one equation into the other to eliminate $M$, then rearrange. Squaring the Chern number condition, notice that
\bea
((2n_b+1)M -1)^2 & = & (2n_b+1)^2M^2 -2(2n_b+1)M + 1\\
& = & (2n_b+1)\left((2n_b+1)M^2 -2M + \frac{1}{2n_b+1}\right)\\
& = & (2(d+e)-1)^2 \\
(2n_b+1)M^2 - 2M & = & 2de 
\eea
Substituting $(2n_b+1)M^2-2M$ to eliminate the variable $M$,
\bea
(2n_b+1)\left(2de+\frac{1}{2n_b+1}\right) & = & 4(d+e)^2-4(d+e)+1\\
(2n_b+1)(2de) +4d + 4e & = & 4d^2+8de+4e^2
\eea
upon which we see the exact same relation as in the $M+1$ case, so the argument is identical.
\end{proof}

With this in hand, we know that the class $E_{n_b}$ determines the graph of $c_b(a)$ at $a=2n_b+1$. We now need to show that this remains the case up until where its obstruction function meets the volume constraint.

\section{Discontinuities of $RF$}
In this section we verify that the points of discontinuity of $n_b$ correspond to roots of a particular quadratic. Many derivative estimates appearing later on will depend on derivatives with respect to $b$, which will only be defined away from these discontinuities. For what follows we will only need $n\geq 4$ but the argument is simpler taking $n\in \N$.

\begin{prop}\label{endpts}
The endpoints of the steps of $n_b$ coincide with the values at $n\in\mathbb{N}$ of the functions
\bea
u(n) & = & n+2-\sqrt{2n+3} \\
l(n) & = & n+1-\sqrt{2n+1}
\eea
Notice that $u(n) = l(n+1)$. Hence, the length of the $n$th step is given by the difference $u(n)-l(n)$, which is
\be
1+(\sqrt{2n+1}-\sqrt{2n+3})
\ee
\end{prop}

\begin{proof}
Define $n_b:\R^+\to\N$ given by $n_b = \floor{b+\sqrt{2b}}$, and $u:\N\to\R^+$ defined above. 

To see that $n_b(u(n)) = n$, first calculate
\bea
n_b(u(n)) & = & n_b(n+1-\sqrt{2n+1})\\
& = & \floor{(n+1-\sqrt{2n+1})+\sqrt{2(n+1-\sqrt{2n+1})}} \\
& = & n+1 +\floor{\sqrt{2(n+1-\sqrt{2n+1})}-\sqrt{2n+1}}
\eea
Setting the term containing the floor function equal to -1, we find that it holds for any value of $n\in \N$. Since this term simplifies to -1, we have $n_b(u(n))=n$ and $n_b$ is surjective from $\{u(n)\}\to\N$.

By the quadratic formula, $b+\sqrt{2b}$ is injective on $\R^+$. So, its restriction $n_b|_{\{u(n)\}}$ is injective as a map to $\N$, and remains so after composing with the floor function which acts as the identity. Thus, the function $b+\sqrt{2b}$ identifies $\{u(n)\}$ with $\N$. 

To see that the $\{u(n)\}$ coincide with the discontinuities of $n_b$, note that the discontinuities of the floor function are precisely the integers. So, having shown that only the $\{u(n)\}$ map to integers under $n_b$, it follows that the $\{u(n)\}$ are the full set of discontinuities.

\end{proof}

\begin{rek}
It is interesting to contrast our result for $b>2$ with the analogous statement for $b=2$, proven in \cite[Lem. 3.10]{CFS}. In their work, the obstructive class determining the $RF$-value of $8\frac{1}{36}$ for $b=2$ is
$F_2 = (6,3;3,2^{\times 7})$.
Now, $b=2$ is a point of discontinuity of $n_b$, and $\lim_{b\to2^+}n_b=4$ so the corresponding class is $E_4 = (4,1;1^{\times 9})$.
When $b=2$ and $a\in[8,9]$ we have $\mu(E_4)(a)\leq\sqrt{\frac{a}{2b}}$ with equality precisely at $a=9$. Since $E_4$ is perfect at $a=9$, for all $b$ $\mu(E_4)(9)$ is the height of the step it could determine in $c_b(a)$. For larger values of $b$ (up until $6-\sqrt{11}$ by Proposition \ref{endpts}) the $RF$-value is no smaller than the ``center" of this step at $a=9$.
Thus $RF(b)>9>8\frac{1}{36}$ for such $b$, explaining the strict inequality $b>2$ in the statement of Theorem \ref{theoremA}.
\end{rek}


\section{The reduction method: $1^{\times 2n_b+2}\subset \w(a)$ and $b>2$}\label{reductiond=1}
Section \ref{2nb+1} showed that $E_{n_b}$ determines $c_b(2n_b+1)$. The obstruction function $\mu_{E_{n_b}}(a)$ is constant with value $\mu_{E_{n_b}}(2n_b+1)$ for $2n_b+1<a<2b\cdot\mu_b(E_{n_b})^2$, and intersects $\sqrt{\frac{a}{2b}}$ at $2b\cdot\mu_b(E_{n_b})^2$. Here we establish that past $2b\cdot\mu_b(E_{n_b})^2$, the function $c_b(a)$ equals the volume constraint. The following bound
from \cite{CFS} limits the range of values we must consider.

\begin{prop}\label{sqrtbound}
For every real $b\geq 2$, we have $c_b(a) = \sqrt{\frac{a}{2b}}$ when $a\geq (\sqrt{2b}+1)^2$.
\end{prop}

So, for our result need only show that $c_b(a)=\sqrt{\frac{a}{2b}}$ up to $a=(\sqrt{2b}+1)^2$. 
Thus, in this section and the next, we consider the interval $[2b\cdot\mu_b(E_{n_b})^2, (\sqrt{2b}+1)^2]$ and aim to show that for any $a\in[2b\cdot\mu_b(E_{n_b})^2, (\sqrt{2b}+1)^2]$, the corresponding weight vector reduces to one with positive defect. In the remainder of the paper, all derivatives are understood to be taken away the set of discontinuities $\{n_b\}$ described in Proposition \ref{endpts}.

We begin with the same initial vector
\be
C_0 = \left((b+1)\lambda;b\lambda, \lambda, \w(a)\right) = \left((b+1)\lambda;b\lambda, \lambda, 1^{\times 2n_b+1}, \w(a-2n_b-1)\right) \nonumber
\ee
which as always has defect $-1$.  We apply a Cremona transformation to obtain the (unordered) vector
\be
C_1 = \left((b+1)\lambda-1;b\lambda-1, \lambda-1, 1^{\times 2n_b}, \w(a-2n_b-1)\right).\nonumber
\ee
Now, the ordering becomes important. The term $b\lambda-1$ is larger than 1 since $b>2$, and is larger than $\lambda-1$ for the same reason. Notice that at its maximum $\lambda=1+\frac{1}{\sqrt{2b}}$, so $2>\lambda$ and hence $\lambda-1$ is less than 1. So the ordering becomes
\be
C_1 = \left((b+1)\lambda-1;b\lambda-1, 1^{\times 2n_b}||\lambda-1, \w(a-2n_b-1)\right)\nonumber
\ee
which has defect
\be
\delta_1 = (b+1)\lambda-1 - b\lambda +1 -2 = \lambda-2.\nonumber
\ee
Since $\lambda-1<1$, this is negative so we apply another Cremona and re-order.
\bea
C_2 & = & \left((b+1)\lambda-1+\lambda-2;b\lambda-1+\lambda-2,1+\lambda-2, 1+\lambda-2,\right.\newline\nonumber\\
& & 1^{\times 2n_b-2}||\left. \lambda-1,  \w(a-2n_b-1)\right)\newline\nonumber\\
& = & \left((b+2)\lambda-3;(b+1)\lambda-3, (\lambda-1)^{\times 2}, 1^{\times 2n_b-2}||\lambda-1,  \w(a-2n_b-1)\right)\newline\nonumber\\
& = & \left((b+2)\lambda-3;(b+1)\lambda-3, 1^{\times 2n_b-2}||\left(\lambda-1\right)^{\times 3},  \w(a-2n_b-1)\right)\newline\nonumber
\eea
The defect here is
\be
\delta_2 = (b+2)\lambda-3-(b+1)\lambda+3-2 = \lambda-2. \nonumber
\ee
At this point, we note that the defect is the same but there are 2 fewer $1$s in the vector. We apply $n_b-1$ more Cremonas, for a total of $n_b$, as this gets rid of $2n_b$ total copies of 1.\newline\indent
Thus we have the unordered vector
\bea
C_{n_b}:\left((b+n_b+1)\lambda-(2n_b+1);(b+n_b)\lambda-(2n_b+1)\right. \nonumber \newline\nonumber\\
\left.\left(\lambda-1\right)^{\times 2n_b+1},  \w(a-2n_b-1)\right). \nonumber\label{initialvector}
\eea
Computing the defect here is again dependent on ordering. 
We will need to work with cases distinguished by their orderings of the $\lambda-1$ term, the $(b+n_b)\lambda - (2n_b+1)$ and the first term of $\w(a-2n_b-1)$. To lighten notation, we abbreviate some of these terms as follows.
\begin{equation}
\begin{cases}
\w(a-2n_b-1) = ((d)^{\times k},(d')^{\times k'},(d'')^{\times k''},\mathcal{W})\\
 \sqrt{\frac{a}{2b}} := \lambda \\
(b+n_b)\lambda - (2n_b+1) := m
\end{cases}
\end{equation}
Here $k^{(n)}$ denotes the multiplicity of the weight expansion term $d^{(n)}$, and $\mathcal{W}$ consists of the remaining terms in the weight expansion.

Here in Section \ref{reductiond=1} we give the arguments when $\w(a-2n_b-1)$ begins with a 1, i.e. $d=1$. This explains the notation used in the title of this section, as there are $2n_b+2$ copies of $1$ in the weight vector. Next, in Section \ref{reductiond<1} we argue when $d<1$, so there are $2n_b+1$ copies of 1 in the weight vector and the first term of this weight expansion is strictly less than 1. We consider each ordering of the subsequent terms as a separate case. Some of these cases are redundant, others are impossible, while still others require further sub-cases to cover all possibilities.

The table below lists cases covered in this section with their section numbers, and Figure~\ref{fig:mdlambdafigure} gives a schematic depiction of the change in each term as $a$ and $b$ vary where $m,\lambda, d'$ are all continuous. 
\begin{center}
\begin{tabular}{|c|c|c|c|c|c}
\hline
Ordering   & Sub-Case & Section  \\
\hline
$1>d'>\lambda-1>m$ & - & \ref{1dlambdam} \\
$1>d'>m>\lambda-1$ & $k'\geq 4$ & \ref{ldlargestk4}\\
$1>d'>m>\lambda-1$ & $k'=3$ & \ref{ldlargestk3}\\
$1>d'>m>\lambda-1$ & $k'=2$ & \ref{ldlargestk2}\\
$1>d'>m>\lambda-1$ & $k'=1$ and $d''\leq m$ & \ref{ldlargestk1md}\\
$1>d'>m>\lambda-1$ & $k'=1$ and $m\leq d''$ & \ref{1dlargestk1dm}\\
$1>m>\lambda-1>d'$ & - & \ref{1mlambdad}\\
$1>m>d'>\lambda-1$ & - & \ref{1mlargestdlambda}\\
$1>\lambda-1>m>d'$ & $2m+1-2(\lambda-1)>(\lambda-1)$ & \ref{1lambdalargestA}\\
$1>\lambda-1>m>d'$ & $2m+1-2(\lambda-1)<(\lambda-1)$ & \ref{1lambdalargestB}\\
$1>\lambda-1>d'>m$ & - & \ref{1lambdadm}\\
\hline
\end{tabular}
\end{center}

\begin{figure}
  \centering
\includegraphics[width=\textwidth]{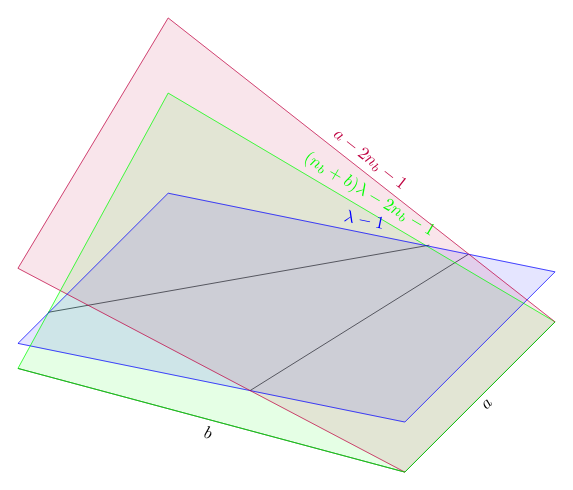}
\caption{Relative ordering of terms in the tail of the weight vector in Equation \ref{initialvector}, restricted to $a\in[RF, (\sqrt{2b}+1)^2]$ and an interval where each such term is smooth in $b$. Black line segments are intersections of the graphs of these terms as functions of $a,b$.}\label{fig:mdlambdafigure}
\end{figure}

There are several estimates for these quantities which we list here for use in the sequel.

\begin{lem}\label{basicineq}
For $n_b=\floor{\sqrt{2b}+b}$, we have
\begin{enumerate}
    \item $b<n_b$ \label{blessnb}
    \item $1<\frac{2n_b+1}{b+n_b}$ and $\lim_{b\to\infty}\frac{2n_b+1}{b+n_b} = 1$ \label{2nb1overbnb}
\end{enumerate}
\end{lem}

\begin{proof}
For (1), $n_b=\floor{b+\sqrt{2b}}$ is bounded below by $b+\sqrt{2b}-1$, which is larger than $b$ once $b>\frac{1}{2}$.

For (2), this is equivalent to $b+n_b<2n_b+1$, or $b<n_b+1$, which follows from (\ref{blessnb}). The remaining claim follows from calculating $\frac{\partial}{\partial b}\left(\frac{2n_b+1}{b+n_b}\right)$:

\bea
\frac{\partial}{\partial b}\left(\frac{2n_b+1}{b+n_b}\right) & = & \frac{-(2n_b+1)(1)}{b+n_b}<0. 
\eea

as $\frac{\partial n_b}{\partial b} = 0$.


\end{proof}

\begin{lem}\label{RFVderivatives}
For $RF(b)$ and $V(b)$ as defined above:
\begin{enumerate}
    \item $\frac{\partial RF}{\partial b} = RF(b)\left(\frac{n_b-b}{b(n_b+b)}\right)$ \label{RFderiv}
    \item $\frac{\partial V}{\partial b} = V(b)\left(\frac{\sqrt{2}}{\sqrt{b}(\sqrt{2b}+1)}\right)$ \label{Vderiv}
    \item $\frac{\partial^2 RF}{\partial b^2},\frac{\partial^2 V}{\partial b^2}< 0$.
    \item $a'(t,b)\leq \frac{a(t,b)}{b}$\label{aderiv}
\end{enumerate}
\end{lem}

\begin{proof}
Direct calculation.

\end{proof}

\begin{lem}\label{dprimeestimates}
For $d'=a-2n_b-1-1$, we have
\begin{enumerate}
    \item When $b\in\{u(n)\}$, $d'=-1$.\label{dprimestart}
    \item $\frac{\partial d'}{\partial b} > 0$ \label{dprimedb}
    \item $\frac{\partial^2 d'}{\partial b} < 0$ \label{dprimedb2}
    \item $\frac{\partial d'}{\partial t}>0$
\end{enumerate}
\end{lem}

\begin{proof}
Direct calculation.
\end{proof}

\begin{lem}\label{mestimates}
Recall that $m = (b+n_b)\lambda - (2n_b+1)$.
\begin{enumerate}
    \item When $a=RF$, $m=0$, and $m=0$ at all $\{u(n)\}$. In general for all $b>1$, $m\geq 0$. and $a\in[RF, (\sqrt{2b}+1)^2]$, $0\leq m<1$.\label{mstart}
    \item $\frac{\partial m}{\partial b} > 0 $ \label{mdb}
    \item $\frac{\partial^2 m}{\partial b^2} >0 $\label{mdb2}
    \item $\frac{\partial m}{\partial t} >0 $ \label{mdt}
\end{enumerate}
\end{lem}

\begin{proof}
Direct calculation.
\end{proof}

\begin{lem}\label{lambdaineq}
For $\lambda=\sqrt{\frac{a(t,b)}{2b}}$ we have
\begin{enumerate}

    \item $\frac{\partial \lambda}{\partial t}>0$ \label{dlambdadt}
    \item $\frac{\partial \lambda}{\partial b}<0$\label{dlambdadb}
    \item $\lambda > 2$ \label{lambdalowerbound}
\end{enumerate}
\end{lem}

\begin{proof}

Direct calculation.
\end{proof}

\subsection{The ordering $1=d>d'>\lambda-1>m$}
\label{1dlambdam}
\be
C_{n_b}:\left(m+\lambda;1,d',\left(\lambda-1\right)^{\times 2n_b+1},m, \mathcal{W} \right).
\ee
The defect is
\bea
\delta_{n_b} & = & m+\lambda-1-d'-\lambda+1\\
& = & m-d'
\eea
which is negative by assumption, so we apply a Cremona transformation.
\bea
C_{n_b+1} & : & \left( m +\lambda+m-d';1+m-d',d'+m-d',\lambda-1+m-d',\left(\lambda-1\right)^{\times 2n_b},\right.\\
& & \left. m,  \mathcal{W}\right).\\
& = & (2m+\lambda-d';1+m-d',m^{\times 2},\lambda-1+m-d',(\lambda-1)^{\times 2n_b}, \mathcal{W})
\eea
After re-ordering, we get
\bea
C_{n_b+1} & : & \left(2m +\lambda-d';1+m-d',m^{\times 2},\left(\lambda-1\right)^{\times 2n_b}, \lambda-1+m-d',\right.\\
& & \left. m, \mathcal{W}\right).
\eea
The defect is
\bea
\delta_{n_b+1} & = & 2m +\lambda-d'-1-m+d'-2m\\
& = & \lambda-1-m
\eea
which is positive by assumption.

\subsection{The ordering $d=1>d'>m>\lambda-1$}
\label{1dmlambda}

\be
C_{n_b} : (m+\lambda;1,d'^{\times k},m,(\lambda-1)^{\times 2n_b+1}||\mathcal{W})\label{1dlargest}
\ee

The defect here depends on $k'$, the multiplicity of $d'$; in particular, whether $k'$ is 2, 3, or 4. In Lemma \ref{dprimemcomparison} below, we show that $k'\geq 5$ is impossible given the ordering above. So, the additional terms in the weight expansion necessitate the following cases:
\begin{enumerate}
    \item $k'= 4$, so $d'\leq \frac{1}{4}$.
    \item $k'=3$, so $\frac{1}{4} < d'\leq \frac{1}{3}$.
    \item $k'=2$ so $\frac{1}{3} < d'\leq \frac{1}{2}$.
    \item $d' > \frac{1}{2}$ and $d''\geq m$, so there is 1 copy of $d'$, 1 copy of $d''$.
    \item $d' > \frac{1}{2}$ and $d''\leq m$, so there is 1 copy of $d'$, 1 copy of $d''$ with $m$ between.
\end{enumerate}




\begin{lem}\label{dprimemcomparison}
If $d'>m$, then $d'>\frac{1}{5}$. 
\end{lem}

\begin{proof}
We look for the values of $b$ and $t$ where $m=d'$. The minimal $d'$ where this occurs will be our lower bound.

By Lemma \ref{mestimates} (\ref{mdb}) and Lemma \ref{dprimeestimates} (\ref{dprimedb}), the first derivatives of $m$ and $d'$ with respect to $b$ are strictly positive. Note also that when $b\in\{u(n)\}$, $d'=-1$ by Lemma \ref{dprimeestimates} (\ref{dprimestart}) and $m=0$ by Lemma \ref{mestimates} (\ref{mstart}). At the same time, the second derivative of $d'$ with respect to $b$ is negative by Lemma \ref{dprimeestimates} (\ref{dprimedb2}) while the second derivative of $m$ with respect to $d$ is positive by Lemma \ref{mestimates} (\ref{mdb2}).

Thus, for all fixed $t$, consider the change in the value $m=d'$. While $d'=-1$ and $m=0$ at the left endpoint of each step, the intersection $m=d'$ occurs closer to the left endpoint as $b$ decreases, since $d'$ grows faster and $m$ grows slower (so the intersection point moves to a smaller value of $b$). Moreover, the value of $d'$ at this smaller value of $b$ is smaller, since $\frac{\partial m}{\partial b}>0$. Thus, to obtain a lower bound on $m=d'$, it suffices to consider the first interval $b\in[2,6-\sqrt{11}]$, i.e. where $b$ is smallest.

Now we examine the equation determining the intersection point between $m$ and $d'$. We differentiate with respect to $t$, to find the $t$-derivatives of both $\frac{\partial m}{\partial b}$ and $\frac{\partial d'}{\partial b}$ and show they are both positive. Thus, as $t$ increases, the value at their intersection point (if it exists) only increases. This means we get a lower bound at the minimal value for the intersection point, which will occur at smallest possible $t$ for which $m=d'$.

In what follows, we let $y$ be the value of $m$ or $d'$. Note that $b$ and $t$ are independent, so $n_b$ is constant in $t$.

\bea
y = m(b,t) = (b+n_b)\sqrt{\frac{a(t,b)}{2b}}-2n_b-1 & & y = d'(b,t) = a(t,b)-2n_b-1-1\\
y + 2n_b+1 = (b+n_b)\sqrt{\frac{a(t,b)}{2b}} & & y+2n_b+2 = a(t,b)\\
\frac{y+2n_b+1}{b+n_b} = \sqrt{\frac{a(t,b)}{2b}} &  & y+2n_b+2=a(t,b)\\
2b\left(\frac{y+2n_b+1}{b+n_b}\right)^2 = a(t,b) &  & y+2n_b+2=a(t,b)\\
2b\left(\frac{y+2n_b+1}{b+n_b}\right)^2 & = & y+2n_b+2
\eea

Now we differentiate this implicitly with respect to $t$.

\bea
\frac{2b}{(b+n_b)^2}\cdot 2\left(y+2n_b+1\right)\cdot\left(\frac{\partial y}{\partial t}\right)^2 = \frac{\partial y}{\partial t}\\
\frac{\partial y}{\partial t}\left(\left(\frac{\partial y}{\partial t}\right)\frac{2b}{(b+n_b)^2}\cdot 2(y+2n_b+1)-1\right) & = & 0
\eea

So either $\frac{\partial y}{\partial t}=0$ uniformly (obviously false) or we have

\bea
0 & = & \left(\frac{\partial y}{\partial t}\right)\frac{2b}{(b+n_b)^2}\cdot 2(y+2n_b+1)-1\\
\frac{\partial y}{\partial t} & = & \frac{(b+n_b)^2}{4b(y+2n_b+1)}
\eea

This is strictly positive so long as $y\geq0$, as needed.

Now, the intersection $m=d'$ does not necessarily happen at $t=0$. However, the above implicit differentiation shows that the minimizing $t$ along the locus $m=d'$ yields the minimal value of $m$ and $d'$ when they coincide.

So, we derive a formula for this intersection point. By the previous arguments we need only calculate on the first step $b\in[2,6-\sqrt{11}]$, on which $n_b=4$. After this simplification, what remains are continuous functions in $b$. 

\bea
& & \lim_{b\to (6-\sqrt{11})^-} m\\
& = & \lim_{b\to (6-\sqrt{11})^-} (b+n_b)\sqrt{\frac{\left((1-t)2b\left(\frac{2n_b+1}{b+n_b}\right)^2+t\left((\sqrt{2b}+1)^2\right)\right)}{2b}}-2n_b-1\\
& = & \lim_{b\to (6-\sqrt{11})^-} (b+4)\sqrt{\frac{\left((1-t)2b\left(\frac{9}{b+4}\right)^2+t\left((\sqrt{2b}+1)^2\right)\right)}{2b}}-9\\
& = & (6-\sqrt{11}+4)\sqrt{\frac{\left((1-t)2(6-\sqrt{11}\left(\frac{9}{6-\sqrt{11}+4}\right)^2+t\left((\sqrt{2(6-\sqrt{11})}+1)^2\right)\right)}{2(6-\sqrt{11})}}-9\\
\eea

Similarly, the formula for $d'$ is

\bea
& & \lim_{b\to (6-\sqrt{11})^-} d'\\
& = & a(t,b)-2n_b-1-1\\
& = & \lim_{b\to (6-\sqrt{11})^-} \left((1-t)2b\left(\frac{2n_b+1}{b+n_b}\right)^2+t\left((\sqrt{2b}+1)^2\right)\right)-2n_b-1-1\\
& = & \lim_{b\to (6-\sqrt{11})^-} \left((1-t)2b\left(\frac{9}{b+4}\right)^2+t\left((\sqrt{2b}+1)^2\right)\right)-10\\
& = & \left((1-t)2(6-\sqrt{11})\left(\frac{9}{6-\sqrt{11}+4}\right)^2+t\left((\sqrt{2(6-\sqrt{11})}+1)^2\right)\right)-10\\
\eea
It is straightforward to find the intersection numerically: this occurs at $m=d'=0.2252$, so certainly if $d'=m$ then $d'>\frac{1}{5}$.

\end{proof}

\subsubsection{Case 1: $k'= 4$, so $\frac{1}{5}<d'\leq \frac{1}{4}$ }\label{ldlargestk4}
In this case the defect is
\be
\delta_{n_b} = m+\lambda-1-2d'
\ee
which is negative under these assumptions. Applying a Cremona transformation we obtain the following unordered vector.

\bea
C_{n_b+1} &:& (m+\lambda+m+\lambda-1-2d';1+m+\lambda-1-2d',(d'+m+\lambda-1-2d')^{\times 2},\\
& & d'^{\times 4-2},m, (\lambda-1)^{\times 2n_b+1}, \mathcal{W})\\
& = & (2m+1+2(\lambda-1)-2d';||m+\lambda-2d', (m+\lambda-1-d')^{\times 2}, d^{\times 4-2},\\
& & m,(\lambda-1)^{\times 2n_b+1},\mathcal{W})
\eea

\begin{lem}\label{dmlesslambda}
Always, $d'-m<\lambda-1$.
\end{lem}
\begin{proof}
We show first that $\frac{\partial}{\partial t}(m+\lambda-1-d')\leq 0$. In this case $d=1$ and $d'=a-2n_b-1-1$.
\bea
& &\frac{\partial}{\partial t}(m+\lambda-1-d')\\
& = &\frac{\partial}{\partial a}((b+n_b)\lambda-2n_b-1+\lambda-1-d')\\
& = & (b+n_b+1)\frac{\partial \lambda}{\partial t} - \frac{\partial d'}{\partial t}\\
& = & (b+n_b+1)\frac{1}{\sqrt{2b}}\left(\frac{1}{2\sqrt{(1-t)RF(b)+tV(b)}}\right)\left(-RF(b)+V(b)\right) - (V(b)-RF(b))\\
\eea
so we hope to show that this is negative. This is equivalent to showing
\bea
0 & > & (b+n_b+1)\frac{1}{\sqrt{2b}2\sqrt{a(t,b)}}-1\\
b+n_b+1 & \leq & 2\sqrt{2ba(t,b)}
\eea
The right-hand side is minimized at $a=RF(b)$. For this value of $a$, the RHS becomes $8b\cdot\frac{2n_b+1}{b+n_b}$, which is bounded below by Lemma \ref{basicineq} (\ref{2nb1overbnb}), and the stronger inequality $b+n_b+1\leq 8b$ is true for $b>1$. Thus $\frac{\partial}{\partial t}(m+\lambda-1-d')\leq 0$.




So, to determine whether $m+\lambda-1-d'\geq 0$, it suffices to verify it at $a_{max}=(\sqrt{2b}+1)^2$, since it is larger for other values of $t$.

\bea
m+\lambda-1-d' & = & (b+n_b)\lambda-(2n_b+1)+\lambda-1-d'\\
& = & (b+n_b)\frac{\sqrt{2b}+1}{\sqrt{2b}}-2n_b-1+\frac{\sqrt{2b}+1}{\sqrt{2b}}-1-((\sqrt{2b}+1)^2-2n_b-2)\\
& = & (b+n_b)\frac{\sqrt{2b}+1}{\sqrt{2b}} + \frac{\sqrt{2b}+1}{\sqrt{2b}}-1-(\sqrt{2b}+1)^2+1\\
& = & (b+n_b+1)\frac{\sqrt{2b}+1}{\sqrt{2b}}-(\sqrt{2b}+1)^2\\
& = & (b+b+\sqrt{2b})\frac{\sqrt{2b}+1}{\sqrt{2b}}-(\sqrt{2b}+1)^2\\
& = & 0
\eea

Now we check that $\frac{\partial}{\partial b}(m+\lambda-1-d')\leq 0$, so that the above is a lower bound as needed.

\bea
& & \frac{\partial}{\partial b}(m+\lambda-1-d')|_{a=(\sqrt{2b}+1)^2}\\
& = & \frac{\partial}{\partial b}\left((b+n_b)\frac{\sqrt{2b}+1}{\sqrt{2b}}-2n_b-1+\frac{\sqrt{2b}+1}{\sqrt{2b}}-1-((\sqrt{2b}+1)^2-2n_b-2)\right)\\
& = & \frac{\partial}{\partial b}\left((b+n_b+1)\frac{\sqrt{2b}+1}{\sqrt{2b}}-\sqrt{2b}+1)^2\right)\\
& = & \frac{\sqrt{2b}+1}{\sqrt{2b}} + (b+n_b+1)\left(-\frac{1}{2}(2b)^{-\frac{3}{2}}\cdot 2\right)\\
& = & \frac{1}{\sqrt{2b}}\left(\sqrt{2b}+1-(b+n_b+1)(4b^2)\right)
\eea
which is clearly negative for $b>1$. 

\end{proof}



To determine the ordering, notice that \be
m+\lambda-1-d' < m+\lambda-2d'
\ee
as $d'<1$. Also,
\be
m+\lambda-2d' > d'
\ee
as $m+\lambda-3d' > m+\lambda-1>0$ since the first two terms are non-negative and $3d'<1$ here. Next,
\be
m+\lambda-1-d'< m 
\ee
since $\lambda-1<d'$ by assumption. Also
\be
m<m+\lambda-2d'
\ee
as $\lambda-1<\lambda-2d'$ here. Lastly, observe that
\be
m+\lambda-1-d' < d'
\ee
since this is equivalent to $m-d'<d'-(\lambda-1)$ and the left-hand side is negative while the right-hand side is positive.

Having checked that the ordering is correct and the terms are non-negative, we are ready to apply a Cremona transformation.

\be
C_{n_b+1}: (2m+1+2(\lambda-1)-2d';d'^{\times 2}, m+\lambda-2d', m,(\lambda-1)^{\times 2n_b+1},(m+\lambda-1-d')^{\times 2},\mathcal{W})
\ee
with defect 
\bea \delta_{n_b+1} & = & 2m+1+2(\lambda-1)-2d'-2d'-m-\lambda+2d'\newline\nonumber\\
& = & m+(\lambda-1)-2d'\nonumber
\eea
which is certainly negative by assumption. So we apply a Cremona transformation.

\bea
C_{n_b+2} &:& (2m+1+2(\lambda-1)-2d'+m+(\lambda-1)-2d';\\
& & m+\lambda-2d'+m-(\lambda-1)-2d',(d'+m+(\lambda-1)-2d')^{\times 2},\\
& & m, (\lambda-1)^{\times 2n_b+1},
(m+\lambda-1-d')^{\times 2}, \mathcal{W})\\
& = & (3m+1+3(\lambda-1)-2d';||2m+1-4d', (m+(\lambda-1)-d')^{\times 2},\\
& & m,(\lambda-1)^{\times 2n_b+1},
(m+\lambda-1-d')^{\times 2},\mathcal{W})
\eea

At this point we again verify that the terms are positive. The head term is

\be
3m+1+3(\lambda-1)-2d'
\ee

and since $k'=4$, we have $1-2d'>\frac{1}{2}$ so this is immediate. The first tail term is non-negative since $d'\leq \frac{1}{4}$. We already showed that $m+\lambda-1-d'>0$.

Now for the ordering. We need only check whether $2m+1-4d' > m$ but we assume here that $d'\leq\frac{1}{4}$, so we get a defect of

\bea
\delta_{n_b+2} & = & 3m+1+3(\lambda-1)-2d'-2m-1+4d'-m-\lambda+1\\
& = & 2(\lambda-1)+2d'
\eea

which is positive.


\subsubsection{Case 2: $\frac{1}{3}\leq d'< \frac{1}{2}$, so $k'= 3$.}\label{ldlargestk3}
We must also deal with the possibility that there are exactly 3 copies of $d'$, in which case we have

\bea
C_{n_b+2} &:& \left((b+n_b+2)\lambda-(2n_b+2+2d'); (b+n_b)\lambda-(2n_b+1),d',\right.\newline\nonumber\\
& &\lambda-2d', (\lambda-1)^{\times 2n_b+1},(\lambda-1-d')^{\times 2}|| \w(1-2d',d'))\nonumber
\eea

with defect 
\bea \delta_{n_b+2} & = & (b+n_b+2)\lambda-(2n_b+2+2d')-(b+n_b)\lambda+(2n_b+1)-d'-\lambda+2d'\newline\nonumber\\
& = & \lambda - 1-d' \nonumber
\eea
which is negative since $d'>\lambda-1$ by assumption. So we apply a further Cremona to obtain

\bea
C_{n_b+3} &:& \left((b+n_b+3)\lambda-(2n_b+3+3d'); (b+n_b+1)\lambda-(2n_b+2+d'),\lambda-1,\right.\newline\nonumber\\
& &\lambda-2d', (\lambda-1)^{\times 2n_b+1},(\lambda-1-d')^{\times 2}|| \w(1-2d',d'))\nonumber
\eea
which reorders to
\bea
C_{n_b+3}&:& \left((b+n_b+3)\lambda-(2n_b+3+3d'); (b+n_b+1)\lambda-(2n_b+2+d'),\right.\newline\nonumber\\
& &\lambda-2d', (\lambda-1)^{\times 2n_b+2},(\lambda-1-d')^{\times 2}|| \w(1-2d',d'))\nonumber
\eea
where possibly the first two terms are switched. Either way, this has defect 
\bea \delta_{n_b+3}& = & (b+n_b+3)\lambda-(2n_b+3+3d')\\
& & -(b+n_b+1)\lambda+(2n_b+2+d')-\lambda+2d'-\lambda+1\newline\nonumber\\
& = & 0.\nonumber
\eea
So again an embedding exists.

\subsubsection{Case 3: $\frac{1}{3}<d'\leq \frac{1}{2}$, so $k'= 2$.}\label{ldlargestk2}
\be
C_{n_b+1} : (m+\lambda;1,(d')^2,m, (\lambda-1)^{\times 2n_b+1},\mathcal{W})\nonumber
\ee
with defect
\bea \delta_{n_b+1} & = & m+\lambda-1-2d'\newline\nonumber\\
& = & m + \lambda -1 -2d'.\nonumber
\eea
This is certainly negative given the assumptions, so we continue with another Cremona:
\bea
C_{n_b+2} &:& \left(m+\lambda+m+\lambda-1-2d';1+m+\lambda-1-2d',m+\lambda -1-d',\right.\newline\nonumber\\ & &  m+\lambda -1-d',m, (\lambda-1)^{\times 2n_b+1}, \mathcal{W}).\nonumber
\eea
By Lemma \ref{dmlesslambda}, we know $d'-m<\lambda-1$, so these terms are all positive.

Now we determine the ordering. Recall that the unordered vector is
\bea
C_{n_b+2} &:& \left(2m+2\lambda-1-2d';m+\lambda-2d',(m+\lambda -1-d')^{\times 2},m, (\lambda-1)^{\times 2n_b+1},\mathcal{W})\right).\nonumber
\eea

We should compare the terms $m, m+\lambda-2d', m+\lambda-1-d'$. Since $d'\leq\frac{1}{2}$ and $d'>m$ here, $\lambda-2d' \geq \lambda-1>0$ as shown above. However, there are two possible relative orderings:
\bea
& & 0<m+\lambda-1-d' < m < m+\lambda-2d'\\
& & 0< m < m+\lambda-1-d' < m+\lambda-2d'
\eea
Also since $d'>m$ here, we must also have $m + \lambda -1 -d' <\lambda-1$. This means
\be
\lambda-1-d'+m < \lambda-1 \leq \lambda -d' -d'+m =  \lambda-2d' + m
\ee
We assumed at the outset that $\lambda-1<m<d'$. But then the only possible ordering is
\bea
& & 0 < m+\lambda-1-d' < \lambda-1 < m < m+\lambda-2d'
\eea

So, the ordered vector is
\bea
C_{n_b+2} &:& \left(2m + 2\lambda - 1 -2d'; m+\lambda -2d',m,\right.(\lambda-1)^{\times 2n_b+1}, \newline\nonumber\\
& & (m+\lambda-1-d')^{\times 2}, (\lambda-1-d')^{\times 2}|| \w(1-2d',d'))\nonumber
\eea
with defect 
\bea \delta_{n_b+2} & = & 2m + 2\lambda -1-2d'-m-\lambda+2d'-m -\lambda+1\newline\nonumber\\
& = & 0\nonumber
\eea
as needed.

\subsubsection{Case 4: $d' > \frac{1}{2}$ and $d''\geq m$}\label{ldlargestk1md}
\be
C_{n_b+1} : (m+\lambda;1,d',d'',m, (\lambda-1)^{\times 2n_b+1},\mathcal{W})\nonumber
\ee
The defect here is
\bea \delta_{n_b+1} & = & m+\lambda-1-d'-d''\newline\nonumber\\
& = & m+\lambda-1-d'-d''\nonumber
\eea
This is non-positive as $d'+d''=1$ by properties of the weight expansion, so we apply a further Cremona.
\bea
C_{n_b+2} &:& (m+\lambda+(m+\lambda-1-d'-d'');\newline\nonumber\\
& & 1+(m+\lambda-1-d'-d''),d'+(m+\lambda-1-d'-d''),\newline\nonumber\\
& & d''+(m+\lambda-1-d'-d''),m, (\lambda-1)^{\times 2n_b+1},\mathcal{W})\newline\nonumber\\
& = & (2m+1+2(\lambda-1)-d'-d'';||m+\lambda-d'-d'',\\
& & m+\lambda-1-d'',m+\lambda-1-d',m, (\lambda-1)^{\times 2n_b+1},\mathcal{W}).
\eea
Lemma \ref{dmlesslambda} shows that the newly generated terms here are positive. To determine the ordering, notice that
\be
m+\lambda-1-d' \leq \lambda-1
\ee
is equivalent to $m\leq d'$, which we assumed was true for this case. A similar argument shows that the remaining two new terms introduced at this step must be smaller than $\lambda-1$, which gives the following vector.
\bea
C_{n_b+1} &:& (2m+1+2(\lambda-1)-d'-d'' ; m+\lambda-d'-d'',\\
& & m, (\lambda-1)^{\times 2n_b+1}, m+\lambda-1-d'',m+\lambda-1-d', \mathcal{W}).
\eea
with defect
\bea 
\delta_{n_b+1} & = & 2m+1+2(\lambda-1)-d'-d''-m-\lambda+d'+d''-m-(\lambda-1)\newline\nonumber\\
& = & 0
\eea
as needed.

\subsubsection{Case 5: $d' \geq \frac{1}{2}$ and $d''\leq m$}\label{1dlargestk1dm}
\bea
C_{n_b+1} &:& (m+\lambda;1,d',m,||d''^{\times k''},(\lambda-1)^{\times 2n_b+1}, \w(1-d',d'))\nonumber
\eea
with defect
\bea \delta_{n_b+1} & = & m+\lambda-1-d'-m\newline\nonumber\\
& = & \lambda -1-d'\nonumber
\eea
This is negative by assumption on the ordering, so we must apply a Cremona.
\bea
C_{n_b+2} &:& ((m+\lambda+(\lambda-1-d'); 1+(\lambda-1-d'),d'+(\lambda-1-d'),m+(\lambda-1-d')\newline\nonumber\\ & & ||d''^{\times k''},(\lambda-1)^{\times 2n_b+1}, \mathcal{W})\newline\nonumber\\
& = & (m+1+2(\lambda-1)-d';||\lambda-d', (\lambda-1)^{\times 2n_b+2}, m+\lambda-1-d',d''^{\times k''},\mathcal{W})\nonumber
\eea
The new terms are positive by the same arguments as in the previous Section \ref{ldlargestk1md}. For the ordering, we have $\lambda-1\geq m+\lambda-1-d'$ as in the previous section, so we must consider the placement of $d''$. We need only determine whether $\lambda-1>d''$ or not. 

If so, the defect is
\bea \delta_{n_b+2}& = & m +1+2(\lambda-1)-d'-\lambda+d'-2(\lambda-1)\\ 
& = & m-(\lambda-1)
\eea
which is positive by assumption.

If not, so $d''\geq \lambda-1$, then the defect is
\bea \delta_{n_b+2}& = & m +1+2(\lambda-1)-d'-\lambda+d'-d''-(\lambda-1)\\ 
& = & m-d''
\eea
which is also positive by assumption.

\subsection{The ordering $1>m>\lambda-1>d'$}\label{1mlambdad}
\bea
C_{n_b+1} &:& (m+\lambda;1,m,(\lambda-1)^{\times 2n_b+1},d'^{\times k'},\mathcal{W})\nonumber
\eea
This has defect
\bea \delta_{n_b+1} & = & m+\lambda-1-m-(\lambda-1)\newline\nonumber\\
& = & 0\nonumber
\eea
as needed.

\subsection{The ordering $1>m>d'>\lambda-1$}
\label{1mlargestdlambda}

\bea
C_{n_b+2}& :& (m+\lambda;1,m,d'^{\times k'}, (\lambda-1)^{\times 2n_b+1}, \mathcal{W})\newline\nonumber
\eea

The defect is

\bea
\delta_{n_b+2} & = & m+\lambda-1-m-d'\newline\nonumber\\
& = & \lambda-1-d'\nonumber
\eea

which is negative since by assumption $d'>\lambda-1$. So we apply another Cremona transformation.
\bea
C_{n_b+2}& :& (m+\lambda+\lambda-1-d';1+\lambda-1-d',m+\lambda-1-d', d'+\lambda-1-d', \\
& &  d'^{\times k'-1} ||(\lambda-1)^{\times 2n_b+1}, \mathcal{W})\newline\nonumber\\
& = & (m+1+2(\lambda-1)-d';||\lambda-d',m+\lambda-1-d',d'^{\times k'-1},\\
& & (\lambda-1)^{\times 2n_b+2},d''^{\times k''}, \mathcal{W}).\nonumber
\eea
Positivity of $m+\lambda-1-d'$ follows from Lemma \ref{dmlesslambda}, and the others are clear. For the ordering, $\lambda-d'>m+\lambda-1-d'$ as $m\leq 1$ by Lemma \ref{mestimates} (\ref{mstart}), and $m+\lambda-1-d'>\lambda-1$ as $m>d'$ by assumption. Also, $\lambda-d'>\lambda-1$ and by assumption $d'>\lambda-1$.  

Another consideration is the placement of $d''$, the next term in the weight expansion. Note first that $d''<\lambda-d'$. By Lemma \ref{lambdaineq} (\ref{lambdalowerbound}), $\lambda-1>0$ always, so then $\lambda-(d'+d'')>0$ as $d'+d''\leq 1$, since $1,d',d''$ are consecutive terms in the weight expansion. 

\subsubsection{$k'=1$, $m+\lambda-1-d'>d''$}
\bea
C_{n_b+2} &:& (m+1+2(\lambda-1)-d';||\lambda-d',m+\lambda-1-d',|| d''^{\times k''}, (\lambda-1)^{\times 2n_b+2}, \mathcal{W})
\eea
The defect with the ordering $d''>\lambda-1$ above is
\bea
\delta_{n_b+2} & = & m+1+2(\lambda-1)-d'-\lambda+d'-m-\lambda+1+d'-d''\newline\nonumber\\
& = & d'-d''
\eea
which is positive.

The other possibility, $\lambda-1>d''$, has defect
\bea
\delta_{n_b+2} & = & m+1+2(\lambda-1)-d'-\lambda+d'-m-\lambda+1+d'-(\lambda-1)\newline\nonumber\\
& = & d'-(\lambda-1)
\eea
which is positive by assumption.

\subsubsection{$k'=1$, $d''>m+\lambda-1-d'$}
\bea
C_{n_b+2} &:& (m+1+2(\lambda-1)-d';\lambda-d',d''^{\times k''},m+\lambda-1-d', (\lambda-1)^{\times 2n_b+2}, \mathcal{W})
\eea
The defect is
\bea
\delta_{n_b+2} & = & m+1+2(\lambda-1)-d'-\lambda+d'-2d''\newline\nonumber\\
& = & m+\lambda-1-2d''.
\eea
This defect only occurs if $k''\geq 2$, in which case it must be positive - by properties of the weight expansion, $2d''\leq d'$, and $m-d'>0$ by assumption.

In the case where $k''=1$, the defect would be
\bea
\delta_{n_b+2} & = & m+1+2(\lambda-1)-d'-\lambda+d'-d''-m-\lambda+1+d'\\
& = & d'-d''
\eea
which is non-negative.

\subsubsection{$m+\lambda-1-d'>d'$, $k'\geq 2$}
\bea
C_{n_b+2}:(m+1+2(\lambda-1)-d';\lambda-d',m+\lambda-1-d',d'^{\times k'-1}, (\lambda-1)^{\times 2n_b+2},d''^{\times k''}, \mathcal{W})
\eea
The defect here is
\bea
\delta_{n_b+2} & = & m+1+2(\lambda-1)-d'-\lambda+d'-m-\lambda+1+d'-d'\newline\nonumber\\
& = & 0\nonumber
\eea
as needed.

\subsection{The ordering $1>\lambda-1>m>d'$}
\label{1lambdalargest}

Note that by \ref{basicineq}, under these assumptions we also have $m-(\lambda-1)<1$ since $-(\lambda-1)<0$ by Lemma \ref{lambdaineq} (\ref{lambdalowerbound}) and $m<1$ by Lemma \ref{mestimates} (\ref{mstart}). 

In this case we see that the ordering is
\be
C_{n_b}:\left(m+\lambda;1,(\lambda-1)^{\times 2n_b+1},m,d'^{\times k},\mathcal{W}\right).
\ee
The defect is
\bea
\delta_{n_b} & = & m+\lambda-1-2\lambda+2\\
& = & m -(\lambda-1)
\eea
which is negative by assumption, so we apply a Cremona transformation.
\bea
C_{n_b+1} & : & \left(m +\lambda+m-(\lambda-1);1+m-(\lambda-1),(\lambda-1 + m - (\lambda-1))^{\times 2},\right.\\
& & ,\left.\left(\lambda-1\right)^{\times 2n_b-1}, m,d'^{\times k}, \mathcal{W}\right).\\
& = & \left(2m+1;m+1-(\lambda-1),m^{\times 3},(\lambda-1)^{\times 2n_b-1},d'^{\times k}, \mathcal{W}\right)
\eea
At this point we must determine the placement of the term $m-\lambda+2=m-(\lambda-1)+1$. This is certainly larger than $m$ as $\lambda<2$. We must decide whether
\bea
m-(\lambda-1)+1 & > & \lambda-1\\
m+1 & > & 2(\lambda-1).
\eea
To show this, notice that at $b=2$ (approaching from the right), these quantities are equal. Here the $a$-value doesn't matter since the interval $[2n_b+1, (\sqrt{2b}+1)^2]$ has length 0 at $b=2$, so we just take it to be $2n_b+1$.
\bea
(b+n_b)\lambda -2n_b-1 +1 & = & 2(\lambda-1)\\
(2+4)\lambda-9+1 & = & 2(\lambda-1)\\
6\sqrt{\frac{9}{4}}-8 & = & 2(\sqrt{\frac{9}{4}}-1)\\
\frac{18}{2}-\frac{16}{2} & = & 3-2
\eea
At the points of discontinuity, approaching from the right, by Lemma \ref{mestimates} (\ref{mstart}) we have $m=0$. At the same time, $\lambda$ decreases by Lemma \ref{lambdaineq} (\ref{dlambdadb}). Moreover, we observe that by Lemma \ref{mestimates} (\ref{mdb}), $\frac{\partial m}{\partial b}>0$ while $\frac{\partial \lambda}{\partial b}<0$ by Lemma \ref{lambdaineq} (\ref{dlambdadb}). So this inequality always holds.

Thus, after re-ordering, we get
\bea
C_{n_b+1} & : &\left(2m+1;m+1-(\lambda-1),(\lambda-1)^{\times 2n_b-1},m^{\times 3},d'^{\times k}, \mathcal{W}\right)
\eea

The defect is
\bea
\delta_{n_b+1} & = & 2m +1-m+\lambda-2-2\lambda+2\\
& = & m-(\lambda-1)
\eea
which is again negative. So we apply another Cremona transformation:
\bea
C_{n_b+2} & : &\left(2m+1+m-(\lambda-1);m-\lambda+2+m-(\lambda-1),(\lambda-1)^{\times 2n_b-3},\right.\\
& & \left.m^{\times 5},d'^{\times k}, \mathcal{W}\right)\\
& = & (3m+1-(\lambda-1);2m+1-2(\lambda-1),(\lambda-1)^{\times 2n_b-3},m^{\times 5},d'^{\times k}, \mathcal{W})
\eea
It turns out that the ordering here between the first two terms can go in either direction. Thus we need more cases. 

\subsubsection{$2m+1-2(\lambda-1)$ larger}\label{1lambdalargestA}
In this case the ordering is as above, so the defect is
\bea
\delta_{n_b+2} & = & 3m +1-(\lambda-1)-2m-1+2(\lambda-1)-2\lambda+2\\
& = & m-(\lambda-1)
\eea
This is still negative by assumption.

\bea
C_{n_b+3} & : &\left(3m+1-(\lambda-1)+m-(\lambda-1);2m+1-2(\lambda-1)+m-(\lambda-1),\right.\\
& & \left.(\lambda-1)^{\times 2n_b-5},m^{\times 7},d'^{\times k}, \mathcal{W}\right)\\
& = & (4m+1-2(\lambda-1);3m+1-3(\lambda-1),(\lambda-1)^{\times 2n_b-5},m^{\times 7},d'^{\times k}, \mathcal{W})
\eea

and again we see variations in ordering. The second variation in ordering is covered in the next subsection. Here we assume that the above holds, and hence the defect is
\bea
\delta_{n_b+3} & = & 4m+1-2(\lambda-1)-3m-1+3(\lambda-1)-2(\lambda-1)\\
& = & m-(\lambda-1)
\eea
again negative. Positivity of entries follows from the assumptions of this section. So, we continue with Cremona transformation.

\bea
C_{n_b+4} & : &\left(4m+1-2(\lambda-1)+m-(\lambda-1);3m+1-3(\lambda-1)+m-(\lambda-1),\right.\\
& & \left.(\lambda-1)^{\times 2n_b-7},m^{\times 7},d'^{\times k}, \mathcal{W}\right)\\
& = & (5m+1-3(\lambda-1);4m+1-4(\lambda-1),(\lambda-1)^{\times 2n_b-7},m^{\times 9},d'^{\times k}, \mathcal{W})
\eea

We'll continue this until we get to just 1 copy of $(\lambda-1)$, so overall $n_b$ times to get

\bea
C_{n_b+n_b} & : &\left((n_b+1)m+1-2(\lambda-1)+m-(\lambda-1);n_bm+1-n_b(\lambda-1),(\lambda-1),\right.\\
& & \left.m^{\times 7},d'^{\times k}, \mathcal{W}\right)\\
& = & ((n_b+1)m+1;n_bm+1-n_b(\lambda-1),(\lambda-1),m^{\times 2n_b+1},d'^{\times k}, \mathcal{W})
\eea

Checking the ordering, we want to know whether
\be
n_bm + 1 -n_b(\lambda-1)>\lambda-1
\ee

and both orderings occur. The defect given the above ordering is
\bea
\delta_{n_b+n_b} & = & (n_b+1)m+1-n_bm-1+n_b(\lambda-1)-(\lambda-1)-m\\
& = & (n_b-1)(\lambda-1)
\eea
which is strictly positive. 

In the other case, all that matters is whether 
\bea
n_bm + 1 -n_b(\lambda-1) & > & m\\
(n_b-1)m + 1 & > & n_b(\lambda-1)
\eea
and both orderings occur. If the $m$ term is larger, then the defect is 
\bea
\delta_{n_b+n_b} & = & (n_b+1)m+1-(\lambda-1)-2m\\
& = & (n_b-1)m+1-(\lambda-1)
\eea
To estimate this, we need
\bea
(n_b-1)m+1 & > & \lambda-1\\
(n_b-1)m & > & \lambda-2
\eea
but the right-hand side is strictly negative by Lemma \ref{lambdaineq} (\ref{lambdalowerbound}), while the left-hand side is no less than 0 by Lemma \ref{mestimates} (\ref{mstart}). So this is positive.

Lastly, if the longer $m$ term is larger, then we get

\bea
\delta_{n_b+n_b} & = & (n_b+1)m+1-(\lambda-1)-n_bm-1+n_b(\lambda-1)-m\\
& = & (n_b-1)(\lambda-1)
\eea
which is also strictly positive.

\subsubsection{$\lambda-1$ larger}\label{1lambdalargestB}
In this case the ordering is
\bea
C_{n_b+2} & : & (3m+1-(\lambda-1);(\lambda-1)^{\times 2n_b-3}||,2m+1-2(\lambda-1),m^{\times 5},d'^{\times k}, \mathcal{W})
\eea
By assumption $\lambda-1 > 2m+1-2(\lambda-1)$, but it turns out this cannot happen.
\begin{lem}\label{lambdaless2m}
The above ordering is impossible; that is, we cannot simultaneously have $3(\lambda-1)>2m+1$ and $d>1$.
\end{lem}

\begin{proof}





Recall that $d=a-2n_b-1$, so the above assumptions are equivalent to 
\bea
\sqrt{\frac{2n_b+2}{2b}} < \frac{4n_b-2}{2b+2n_b-3}.
\eea
However, it turns out that the opposite inequality holds. Equivalently, the bound we want is
\be
\frac{2n_b+1}{2b} > \left(\frac{4n_b-2}{2b+2n_b-3}\right)^2.
\ee
This is equivalent to the inequality
\be
8n_b^3+8b^2(n_b+1)+18 > 16n_b^2+8b(2n_b^2-3n_b+4)+6n_b
\ee
at which point we apply the lower and upper bounds for the floor function in $n_b$ to obtain the stronger inequality
\bea
& & 8(b+\sqrt{2b}-1)^3+8b^2(b+\sqrt{2b})+18\\
& > & 16(b+\sqrt{2b}+1)+8b(2(b+\sqrt{2b}+1)^2-3(b+\sqrt{2b}+1)+4)+6(b+\sqrt{2b}+1)
\eea
which simplifies to
\be
504b^{3}-40b^{2}+(168\sqrt{2}b^{2}-54\sqrt{2}b-19\sqrt{2})\sqrt{b}-22b-23>0
\ee
and this holds true for all $b>1$. 
\end{proof}

\subsection{The ordering $1>\lambda-1>d'>m$}\label{1lambdadm}
In this case we see that the ordering is
\be
C_{n_b}:\left(m+\lambda;1,(\lambda-1)^{\times 2n_b+1},d'^{\times k},m,\mathcal{W}\right).
\ee
The defect is
\bea
\delta_{n_b} & = & m+\lambda-1-2\lambda+2\\
& = & m -(\lambda-1)
\eea
which is negative by assumption, so we perform another Cremona.
\bea
C_{n_b} & : & \left(m+\lambda+m -(\lambda-1);||1+m -(\lambda-1),((\lambda-1)+m -(\lambda-1))^{\times 2},\right.\\
& & \left.(\lambda-1)^{\times 2n_b-1},d'^{\times k},m,\mathcal{W}\right)\\
& = & (2m+1; (\lambda-1)^{\times 2n_b-1}, d'^{\times k}, m^{\times 3}, \mathcal{W})
\eea
The defect is
\bea
\delta & = & 2m+1 -3(\lambda-1)
\eea
However, if this is negative, we cannot also have $d>1$ by Lemma \ref{lambdaless2m} in the previous section. Thus we are finished.

\section{The reduction method: $1^{\times 2n_b+2}\not\subset \w(a)$ and $b>2$}\label{reductiond<1}
Here we cover the second branch of possibilities, where the weight expansion terms are no larger than 1. Thus $d\leq 1$ but all the same relative orderings of the remaining terms could occur. The table below summarizes the cases needed.

\begin{center}
\begin{tabular}{|c|c|c|c|c|c}
\hline
Ordering     & Case & Section  \\
\hline
$d>m>\lambda-1$ & $k'\geq3$ & \ref{dmlambdakp3}\\
$d>m>\lambda-1$ & $k'=2$ & \ref{dmlambdakp2}\\
$d'>m>\lambda-1$ & $k'=1$, $m\leq d'$ & \ref{dmlambdak1mdp}\\
$d'>m>\lambda-1$ & $k'=1$, $d'\leq m$ & \ref{dmlambdak1dpm}\\
$d'>\lambda-1>m$ & - & \ref{dlambdam}\\
$m>d>\lambda-1$ & - & \ref{mdlambda}\\
$m>\lambda-1>d$ & - & \ref{mlambdad}\\
$\lambda-1>d>m$ & - & \ref{lambdadm}\\
$\lambda-1>m>d$ & - & \ref{lambdamd}\\
\hline
\end{tabular}
\end{center}

\subsection{The ordering $d>m>\lambda-1$}\label{dmlambda}
\subsubsection{Case 1: $d\leq \frac{1}{3}$, so $k\geq 3$.}\label{dmlambdakp3}
In this case, remaining terms in the weight expansion do not contribute.
\be
C_{n_b+1} : (m+\lambda;d^{\times k}||m,(\lambda-1)^{\times 2n_b+1},\mathcal{W})\nonumber
\ee
with defect
\bea \delta_{n_b+1} & = & m+\lambda-3d'\newline\nonumber\\
& = & \lambda -3d'.\nonumber
\eea
This is non-negative as $\lambda-1>0$ and $3d\leq 1$ by assumption.\newline\indent

\subsubsection{Case 2: $\frac{1}{3}<d\leq \frac{1}{2}$, so $k=2$}\label{dmlambdakp2}
In this case we must consider the effect of $d'$ terms. If $d'>m$, then
\be
C_{n_b+1} : (m+\lambda;d,d,d'^{\times k'},m,(\lambda-1)^{\times 2n_b+1},\mathcal{W})\nonumber
\ee
with defect
\bea \delta_{n_b+1} & = & m+\lambda-2d-d'\nonumber.
\eea

Since $d'=1-2d = 1-2(a-2n_b-1)$ 
\bea
2d+d' & = & 2(a-2n_b-1)+1-2(a-2n_b-1)\\
& = & 1
\eea

so this defect is positive as $\lambda-1>0$.

If $d'<m$, then
\be
C_{n_b+1} : (m+\lambda;d,d,m,d'^{\times k'},(\lambda-1)^{\times 2n_b+1},\mathcal{W})\nonumber
\ee
with defect
\bea
\delta_{n_b+1} & = & m+\lambda-2d-m=\lambda-2d\nonumber.
\eea
which is positive by the argument of \ref{dmlambdakp3}.

\subsubsection{Case 3: $d \geq \frac{1}{2}$ and $d'\geq m$}\label{dmlambdak1mdp}
\be
C_{n_b+1}:(m+\lambda;d,d'^{\times k'},m, (\lambda-1)^{\times 2n_b+1},\mathcal{W})\nonumber
\ee
and with the assumed ordering, we have defect
\bea \delta_{n_b+1} & = & m+\lambda-2d-d'\newline\nonumber\\
& = & m+\lambda-2d-d'.\nonumber
\eea
which is again positive by the argument of \ref{dmlambdakp3}.

\subsubsection{Case 4: $d \geq \frac{1}{2}$ and $d'\leq m$}\label{dmlambdak1dpm}
\be
C_{n_b+1} : (m+\lambda;d,m||d'^{\times k'},(\lambda-1)^{\times 2n_b+1}, \w(1-d',d'))\nonumber
\ee
It now matters whether or not $\lambda-1$ appears as the third term in the defect. When $\lambda-1>d'$, the defect is
\bea
\delta_{n_b+1} & = & m+\lambda-d'-m-(\lambda-1)\newline\nonumber\\
& = & 1-d'.\nonumber
\eea
which is positive as $d'< 1$.\newline\indent
Similarly if $d'>\lambda-1$ the defect is
\bea \delta_{n_b+1} & = & m+\lambda-d-m-d'\newline\nonumber\\
& = & \lambda -d'-d'\nonumber
\eea
which is positive by the argument of \ref{dmlambdakp3}.

\subsection{The ordering $d>\lambda-1>m$}\label{dlambdam}
In this case we see that the ordering is
\be
C_{n_b}:\left(m+\lambda;d^{\times k},(\lambda-1)^{\times 2n_b+1},m,\mathcal{W}\right).
\ee
The defect is
\bea
\delta_{n_b} & = & m+\lambda-3(\lambda-1)\\
& = & m+1 - 2(\lambda-1)
\eea
However, $m+1>2(\lambda-1)$ for all $b>2$ by the lemma below, so this is positive. Thus we are finished.

\begin{lem}\label{mlarger2lambda}
For $b>2$, $m-2(\lambda - 1)+1>0$.
\end{lem}
\begin{proof}
This inequality is equivalent to $m+1>2(\lambda-1)$. When $a=RF$, we have $m=0$, so this reduces to $\frac{1}{2}>\lambda-1$ which holds for all $b>2$. So it suffices to show that $\frac{\partial m}{\partial t}>2\frac{\partial \lambda}{\partial t}$, so the inequality still holds.\newline\indent
We computed $\frac{\partial m}{\partial t}$ in Lemma \ref{mestimates} (\ref{mdt}) and $\frac{\partial \lambda}{\partial t}$ in Lemma \ref{lambdaineq} (\ref{dlambdadt})
so the claim becomes
\bea
& & (b+n_b)\frac{1}{\sqrt{2b}}\left(\frac{1}{2\sqrt{(1-t)RF(b)+tV(b)}}\right)\left(-RF(b)+V(b)\right)\\
& > & 2\cdot\frac{1}{\sqrt{2b}}\left(\frac{1}{2\sqrt{(1-t)RF(b)+tV(b)}}\right)\left(-RF(b)+V(b)\right)
\eea
which is in fact true for all $b>1$.
\end{proof}

\subsection{The ordering $m>d>\lambda-1$}\label{mdlambda}
It will turn out that the ordering
\be
C_{n_b}:\left(m+\lambda;m,d||\left(\lambda-1\right)^{\times 2n_b+1},  \mathcal{W}\right)\label{mbig}
\ee
never happens for any $a,b$ in the intervals of interest. 
\begin{lem}
For any $a\in [RF, (\sqrt{2b}+1)^2]$ and $b\not\in\{u(n)\}$, we have $m<d$.
\end{lem}
\begin{proof}
We show that for $\lambda=\sqrt{\frac{a}{2b}}$,

\bea
(b+n_b)\lambda - 2n_b-1 & < & (1-t)RF(b)+tV(b)-2n_b-1\\
(b+n_b)\lambda & < & (1-t)RF(b)+tV(b)\\
(b+n_b)\sqrt{\frac{a}{2b}} & < & a\\
\frac{b+n_b}{\sqrt{2b}} & < & \sqrt{a}\\
\eea

Minimizing the right-hand side at $a=RF$, we obtain

\bea
\frac{b+n_b}{\sqrt{2b}} & \leq & \sqrt{a}\\
\frac{b+n_b}{\sqrt{2b}} & \leq & \sqrt{2b}\frac{2n_b+1}{b+n_b}\\
\frac{b+n_b}{\sqrt{2b}} & \leq & \sqrt{2b}\frac{2n_b+1}{b+n_b}\\
(b+n_b)^2 & \leq & 2b(2n_b+1)
\eea

This is saturated at $\{u(n)\}$. To prove the inequality, then, we differentiate both sides.


\bea
2(b+n_b) & < & 2(2n_b+1)\\
b+n_b & < & 2n_b+1
\eea

which is immediate from Lemma \ref{basicineq} (\ref{blessnb}).

\end{proof}
It follows that for any $a, b$ where $(b+n_b)\lambda - (2n_b+1)\neq a-2n_b-1$, the ordering in Section \ref{mbig} does not occur.

\subsection{The ordering $m>\lambda-1>d$}\label{mlambdad}
See Section \ref{mbig}.

\subsection{The ordering $\lambda-1>d>m$}
\label{lambdadm}
Now suppose $\lambda-1$ is the largest term. In this case, the weight vector looks like
\bea
C_{n_b} &:& \left((b+n_b+1)\lambda-(2n_b+1);\left(\lambda-1\right)^{\times 2n_b+1}, \right.\newline\nonumber\\
& &\left. (b+n_b)\lambda-(2n_b+1)\right), \w(a-2n_b-1))\nonumber
\eea
This has defect
\bea
\delta_{n_b} & = & (b+n_b+1)\lambda-(2n_b+1)-3\lambda+3\newline\nonumber\\
& = & (b+n_b-2)\lambda -(2n_b-2) = m -2\lambda+3. \nonumber
\eea
To check the sign of this defect, we re-arrange into $m-2(\lambda -1)+1$. By Lemma \ref{mlarger2lambda}, this is positive.

\subsection{The ordering $\lambda-1>m>d$}\label{lambdamd}
The argument of Section \ref{lambdadm} applies here as well.

\section{Obstructive classes when $b\in\{\frac{n+1}{n}\}$ and $a=8$}\label{bnsequence}
At this point, we turn our attention towards the smaller values of $b$, with an eye towards the proof of Theorem \ref{theoremB}. 


\begin{prop}
Let $b_n=\frac{n+1}{n}$. If $n\geq 9$, then $RF(b_n)\leq 9$.\newline\indent
\end{prop}
\begin{proof}
Let $a\geq 9$. As usual, our beginning weight vector has the form
\be
\left((b+1)\lambda;b\cdot\lambda, \lambda, 1^{\times 9}, \w(a-9)\right).\nonumber
\ee
and $\delta = -1$ in the first step, so we see:
\be
\left((b+1)\lambda - 1;b\cdot\lambda -1, \lambda - 1, 1^{\times 8}, \w(a-9)\right).\nonumber
\ee
Note that $\lambda-1 = \sqrt{\frac{a}{2b}} - 1>1$ for $a\geq9$, as this is equivalent to

\bea
1 & < & \sqrt{\frac{a}{2b}}-1\\
2 & < & \frac{3}{\sqrt{2b}}\\
\sqrt{2b} & < & \frac{3}{2}\\
2b & < & \frac{9}{4}\\
b & < & \frac{9}{8}
\eea

If $n \geq 9$ in $b_n = \frac{n+1}{n}$, then this holds, so the defect is
\be
\delta = (b+1)\lambda-1-b\lambda+1-\lambda+1-1 = 0 \nonumber
\ee
guaranteeing an embedding. 
\end{proof}

It follows that we need only consider $a$-values up to $9$. 

\subsection{Eliminating Possible Classes on $(8,9)$}
At this point, our argument essentially follows the lines of \cite{CFS}: for the embedding problem $E(1,8)\hookrightarrow P(\lambda, \lambda\cdot \frac{n+1}{n})$, we find solutions $(d,e;m)$ of non-negative integers to the Diophantine equations  (\ref{dioph1}) and (\ref{dioph2}), subject to the constraint given by Theorem \ref{classcondition}.
To do this, we restrict the set of possible obstructive classes on $(8,9)$. This is accomplished by an analogue of \cite[Prop. 5.2.1]{McSc} or \cite[Prop. 3.6]{CFS}. \\

\begin{prop}\label{no89classes}
For $b_n = \frac{n+1}{n}$ ($n\geq 8$), there are no exceptional classes $(d,e;\m)$ above the volume constraint on $(8,9)$ such that $\ell(a)=\ell(\m)$.
\end{prop}

In the following section, we then identify the necessary classes at $a=8$ to prove Theorem \ref{theoremB}, and show they satisfy the appropriate conditions.


We will use Lemma \ref{errorestimates} (4) to bound the value of $e$. Then, using computer programs adapted from those in \cite{McSc, FrMu, BPT}, we enumerate possible classes $(d,e;\textbf{m})$ with $e$ no greater than this bound. It is then straightforward to verify that none of these classes are obstructive on the interval in question, giving the desired result.\newline\indent

\begin{proof}[Proof of Proposition \ref{no89classes}]

Now we use the estimates of Lemma \ref{errorestimates} (\ref{vmsigmabounds}) to obtain
\bea
v_M\in\left[\frac{1}{3}, \frac{1}{2}\right] & \rightarrow & \frac{\sigma'}{v_M} \leq \frac{3}{2}\newline\nonumber\\
v_M\in\left[\frac{1}{2}, \frac{2}{3}\right] & \rightarrow & \frac{\sigma'}{v_M} \leq \frac{14}{9}\newline\nonumber\\
v_M\geq\frac{2}{3} & \rightarrow & \frac{\sigma}{v_M} \leq \frac{3}{2}\newline\nonumber\\
\eea
Note that in the first two cases we may replace $\sigma'$ by $\sigma$ since $v_M<1$. Then for fixed $q$ and $h$, the above estimates allow us to define the following functions from (\ref{ebound}), replacing $\sigma$ and $v_M$. For $F$, we use that $\sigma<1$ since $\langle \epsilon, \epsilon\rangle<1$. For $G$, we use the largest upper bound above for $\frac{\sigma}{v_M}$.
\bea
F(a,q,h) & := & \frac{\sqrt{2ba}}{\delta}\left(\sqrt{q} - (1-h(1-\frac{1}{b})\right)\\
G(a,q,h) & := &  \frac{\sqrt{2ba}}{\delta}\left(\frac{14}{9\delta}-\left(1-h(1-\frac{1}{b}\right)\right)
\eea
Let $f(q,h):= F(8\frac{1}{q},q,h)$ and $g(q,h):=G(8\frac{1}{q},q,h)$. We choose these values for $a$ because $F,G$ are decreasing on $a\in(8,9)$ (recall that $\delta=y(a)-\frac{1}{q}$ depends on $a$ and $q$).\newline\indent

Then we have by Lemma \ref{degfacts}
\be
2be+h\leq f(q,h)\leq g(q,h).\nonumber
\ee

\begin{lem}\label{fgestimates}
For $b\in(1,\frac{10}{9}]$ and $|h|<\sqrt{2b}$, we have the following. 
\begin{enumerate}
\item $1<q_0<4$, where $q_0$ is the $q$-value where $f(q,h)= g(q,h)$\label{qbound}
\item $\frac{\partial f}{\partial q}>0$ and $\frac{\partial g}{\partial q}<0$ for $q>1$
\item $\frac{\partial f}{\partial h}, \frac{\partial g}{\partial h} > 0$
\item $g(2, h)<6.$
\end{enumerate}
\end{lem}
\begin{proof}
For (1), note that $f(q,h)$ and $g(q,h)$ are equal if and only if $\sqrt{q}=\frac{14}{9}\frac{1}{\delta(q,b)}$, which amounts to 
\be
\sqrt{q}=\frac{14}{9}\cdot \frac{1}{9-\sqrt{8\frac{1}{q}}\frac{2(b+1)}{\sqrt{2b}}}.\nonumber
\ee
This is a quadratic in $q$, and since $b\in (1,1\frac{1}{8}]$, we find that $q$ ranges between 3.33 and 3.654. Hence the intersection point $q_0$ is no larger than 3.\newline\indent

(2) is a straightforward computation from the following formulas.

\bea
\frac{\partial f}{\partial q} & = & \frac{\partial}{\partial q}\left(\frac{\sqrt{2ba}}{\delta}\left(\sqrt{q} - (1-h(1-\frac{1}{b})\right)\right)
\eea

\bea
\frac{\partial g}{\partial q} & = & \frac{\partial}{\partial q}\left(\frac{\sqrt{2ba}}{\delta}\left(\frac{14}{9\delta}-\left(1-h(1-\frac{1}{b}\right)\right)\right)
\eea

(3) is immediate as the coefficients of $h$ in $f, g$ are strictly positive for all $q, b$. Moreover, $\frac{\partial f}{\partial h}, \frac{\partial g}{\partial h}$ are constant in $h$, so as $h$ changes, $f,g$ increase linearly. \newline\indent
For (4), by (3) we simply evaluate at the largest possible value of $h$, which is $h=\sqrt{2b}$. This gives the bound.
\end{proof}

So, by (\ref{qbound}) in Lemma \ref{fgestimates} above, for $q\geq 4$, the inequality (\ref{ebound}) fails, and hence no class can be obstructive at those $q$-values. It remains only to check those values $q=2,3$. We have for all $h$ that
\be
f(3, h), g(3,h) < 6
\ee
so we must verify that for $e\in\{1,2,3,4,5\}$ all the classes with $e$ in this range fail to be obstructive. This can be done with a direct calculation using the obstruction functions.

\begin{table}[h]
\begin{tabular}{|l|ll}
\cline{1-1} \cline{3-3}
Candidate classes for $q=2$                         & \multicolumn{1}{l|}{} & \multicolumn{1}{l|}{Candidate classes for $q=3$}                         \\ \cline{1-1} \cline{3-3} 
$(4,4;4, 3, 1^{\times 8})$                          & \multicolumn{1}{l|}{} & \multicolumn{1}{l|}{$(5,1;1^{\times 11})$}                               \\ \cline{1-1} \cline{3-3} 
$(5,4;4, 3, 3, 1^{\times 7})$                       & \multicolumn{1}{l|}{} & \multicolumn{1}{l|}{$(5,4;4^{\times 2}, 1^{\times 9})$}                  \\ \cline{1-1} \cline{3-3} 
$(5,5;4, 4, 3, 2, 1^{\times 6})$                    & \multicolumn{1}{l|}{} & \multicolumn{1}{l|}{$(5,5;5, 3^{\times 2}, 1^{\times 8})$}               \\ \cline{1-1} \cline{3-3} 
$(5,5;5, 3, 2^{\times 3}, 1^{\times 5})$            & \multicolumn{1}{l|}{} & \multicolumn{1}{l|}{}               \\ \hline
\end{tabular}
\end{table}

To show some sample computations, let $q=2$ and $e=4$. Then the computer search verifies that there is a single class with the same length 10 as the weight expansion of $8\frac{1}{2}$, namely

\be
(4,4;4,3, 1^{\times 8}).
\ee

Its obstruction function at $a=8\frac{1}{2}$ is $\frac{14}{4+4b}$. However, for $b\in[1,2]$, we always have $\frac{14}{4+4b}<\sqrt{\frac{8\frac{1}{2}}{2b}}$, so this class is not obstructive.

Similarly, let $q=3$ and $e=1$. Then there is a single class of length 11 which matches the weight expansion of $8\frac{1}{3}$ (and that of $8\frac{2}{3}$), namely
\be
(5,1;1^{\times 11}).
\ee

Its obstruction function at $a=8\frac{1}{3}$ is $\frac{9}{5+b}$. However for $b\in[1,2]$ we always have $\frac{9}{5+b}<\sqrt{\frac{8\frac{1}{3}}{2b}}$, so this class is not obstructive. At $a=8\frac{2}{3}$, the obstruction function is $\frac{9\frac{1}{3}}{5+b}$, and also $\frac{9\frac{1}{3}}{5+b}<\sqrt{\frac{8\frac{2}{3}}{2b}}$.

Continuing through the list of possible classes, we see that none of these are obstructive for $b\in[1,2]$.

\end{proof}

\subsection{Obstructive Classes at $a=8$, $b_n = \frac{n+1}{n}$}
To establish the $RF$-value of $b_n = \frac{n+1}{n}$, we now show that for each $n$ there is only one possible obstructive class at $a=8$. We define the following infinite family of classes
\bea
R_n & := & ((2n+1)(n+1),(2n+1)n; \frac{1}{8}(2(2n+1)^2 +6), \frac{1}{8} [(2(2n+1)^2+6)-1]^{\times 7}) \\ 
& = & ((2n+1)(n+1),(2n+1)n; n^2+n+1,  (n^2+n)^{\times 7}).
\eea
Changing coordinates to those of $X_n$, this becomes:
\begin{equation}
(3n^2+3n; n^2 + 2n, (n^2+n)^{\times 7},n^2-1).\label{rnconverted}
\end{equation}
If this vector can be reduced to $(0;-1,0, \ldots)$ via  Cremona transforms, then the class will be effective.

\begin{lem}
For $n>2$, the obstructive class $R_n$ from (\ref{rnconverted}) reduces to $(0;-1)$ after $4n+1$ Cremona transforms. 
\end{lem}

\begin{proof}

After 11 Cremonas, the vector (\ref{rnconverted}) becomes

\[
C_{11} = (3n^2 - 13n + 14; ((n-2)^2)^{\times 5}, (n-2)^2 - 1, (n-2)(n-3)^{\times 3}).
\]

Note that at $n=2$, the minimal $n$ for which this sequence makes sense, this vector is already $(0;-1)$ as needed. 

Once $n>2$, we claim that $4n+1$ total Cremona transformations are needed to obtain $(0;-1)$. To see this, we induct on $n$. The base case follows from direct computation. For the induction, we observe that applying 5 Cremona transformations to

\bea
& & (3(n+1)^2+3(n+1);(n+1)^2+2(n+1),((n+1)^2+n+1)^{\times 7}, (n+1)^2-1)\\
& = & (3n^2+9n+6;n^2+4n+3,(n^2+3n+2)^{\times 7}, (n^2+2n))\\
\eea

results in the same vector as applying one Cremona transformation to

\bea
& & (3n^2+3n;n^2+2n,(n^2+n)^{\times 7}, n^2-1)\\
\eea

so the result follows.

\end{proof}
Finally, we show that the $R_n$ are the only possible classes determining the $RF$-value at $a=8$.
\begin{prop}
The only exceptional class with $\mu_{b_n}(8)>\sqrt{\frac{8}{2\frac{n+1}{n}}}$ is
\be
((2n+1)n, (2n+1)(n+1); n^2+n+1, (n^2+n)^{\times 7}).\nonumber
\ee
\end{prop}
\begin{proof}
As in \cite[Lemma 3.10]{CFS}, the strategy is to examine the Diophantine equations for Chern number $1$ and self-intersection $-1$, and show that these equations have no solutions for the given parameter values. The fact that our parameter values are variable is only a technical complication. \newline\indent
Recall that by \hyperref[errorbound]{Theorem \ref{degfacts} (3)}, for an obstructive class $(d,e;\m)$ we have
\be
d=\frac{n+1}{n}e+h \nonumber
\ee
where $d,e\in\N$ and $|h|<\sqrt{\frac{2n+2}{n}}$. It follows that $h\in \frac{1}{n}\Z$, so we write
\be
d=\frac{(n+1)e}{n}+\frac{k}{n}\nonumber
\ee
or, writing as a Diophantine equation,
\be
nd-(n+1)e = k.\label{diok}
\ee
Since $gcd(n,n+1) = 1$ there are integer solutions of the following form. For the specific equation
\be
xn+y(n+1)=gcd(n,n+1) = 1 \nonumber
\ee
we have $x=-1, y=1$ as solution, and more generally a particular solution $(d,e)$ to \ref{diok} is
\be
n(-k)-(-k)(n+1) = k. \nonumber
\ee
It follows from general theory of linear Diophantine equations that all integer solutions can be constructed from this particular solution as
\begin{eqnarray}
d & = & -k+\frac{-(n+1)l}{gcd(n,n+1)} = -k-(n+1)l\newline \label{leqn1}\\
e & = & -k-\frac{nl}{gcd(n,n+1)} = -k-nl \label{leqn2}
\end{eqnarray}
with $l \in \Z$.  We will show that $k=0$ necessarily for an obstructive class of this form. Using again the fact that we are at $a=8$, we apply Lemma \ref{oneblock} to show that the sum of the $m_i$ for any tail $\m$ must satisfy
\[
  2de+1 =
  \begin{cases}
                                   8m^2+2m+1 & \text{if}\,(m+1,m^{\times 7}) \\
                                   8m^2-2m+1 & \text{if}\, (m^{\times 7}, m-1)
  \end{cases}
\]
and similarly
\[
  2(d+e)-1 =
  \begin{cases}
                                   8m+1 & \text{if}\,(m+1,m^{\times 7}) \\
                                   8m-1 & \text{if}\, (m^{\times 7}, m-1)
  \end{cases}
\]
so we combine these to obtain pairs of Diophantine equations. We treat the case $(m-1, m^{\times 7})$ first.




Now we can substitute (\ref{leqn1}) and (\ref{leqn2}) to see
\bea
2(-k-l(n+1) - k -n) -1 & = & 8m-1\newline\nonumber\\
2(k^2-k(2n+1)+n(n+1)) + 1 & = & 8m^2-2m+1.\nonumber
\eea
Solving for $m$ in the first equation and substituting into the first gives the following polynomial in $l$:
\be
l^2\left(n(n+1)-\frac{1}{4}(2n+1)^2\right) + l\left(\frac{-1}{4}(2n+1)\right)+\frac{1}{2}k = 0. \nonumber
\ee
This quadratic has roots
\be
l=-n-\frac{1}{2}\pm\sqrt{(2n+1)^2+\frac{1}{2}k}. \nonumber
\ee

Now, $l$ must be an integer by properties of the solutions to linear Diophantine equations, and $n$ is an integer. Hence we must have $\sqrt{(2n+1)^2+\frac{1}{2}k}\in \frac{1}{2}\Z$, so $\sqrt{4(2n+1)^2+2k}\in \Z$.

Now, $\left|\frac{k}{n}\right|<\sqrt{2\frac{n+1}{n}}$ from \hyperref[errorbound]{Lemma 5.3 (3)}, so $|k|<\sqrt{2n(n+1)}$.

We show that if $k\neq 0$, $4(2n+1)^2+2k$ is never a perfect square, by estimating the distance between $4(2n+1)^2=(2n+2)^2$ and $(2n+3)^2$, which is $4n^2+12n+9-(4n^2+8n+4)=4n+5$. So $2k$ must at least be that large. But $2k <2\sqrt{2n(n+1)}$, and the right-hand side is strictly less than $4n+5$ for all $n$. So this is impossible, and $k=0$.

But then it follows that 
\be
l=n-\frac{1}{2}\pm(2n+1) \nonumber
\ee
which is certainly not an integer. Thus, we cannot see obstructive classes of this form.\newline\indent
On the other hand, we can consider classes of the form $(m+1, m^{\times 7})$. The same substitution gives the quadratic in $l$
\be
l^2\left(n(n+1)-\frac{1}{4}(2n+1)^2\right) + l\left(\frac{1}{4}(2n+1)\right) - \frac{1}{2}k = 0.\nonumber
\ee
Again, since $l$ must be an integer, we see that $k=0$, in which case this polynomial has roots
\be
l=\frac{(2n+1)\pm 4\sqrt{\frac{1}{16}(2n+1)^2-\frac{1}{2}k}}{2} = 0, 2n+1.\nonumber
\ee
Hence the only possible obstructive classes at $a=8$ are of the form
\be
((2n+1)n, (2n+1)(n+1); n^2+n+1, (n^2+n)^{\times 7})\nonumber
\ee
as needed.
\end{proof}

\bibliographystyle{amsplain}

\end{document}